\documentclass[10pt,twoside,reqno]{amsart}
\usepackage[top=1.4in, bottom=1.4in, left=1in, right=1in]{geometry}
\usepackage{graphicx}
\usepackage{amsmath,amscd,amsfonts,amsthm,amssymb}
\usepackage{mathtools,latexsym,bbm,esint,commath,enumerate}
\usepackage{stmaryrd,fancyhdr,indentfirst,hyperref}
\usepackage{mathrsfs,float,caption, subcaption}
\numberwithin{figure}{section}
\usepackage[bottom]{footmisc}
\usepackage{xcolor}

\numberwithin{equation}{section}
\newtheorem{theorem}{Theorem}[section]
\newtheorem{lemma}[theorem]{Lemma}

\theoremstyle{definition}

\def\bax#1\eax{\begin{align*}#1\end{align*}}
\def\bal#1\eal{\begin{align}#1\end{align}}
\def\bas#1#2\eas{\begin{align}\label{#1}\begin{split}#2\end{split}\end{align}}

\newcommand{\Rm}[1]{
  \textup{\uppercase\expandafter{\romannumeral#1}}
}

\newcommand{\diff}{\,\mathrm{d}}
\newcommand{\Dissip}{\mathcal{D}}

\newcommand{\N}{\mathbb{N}}

\newcommand{\pt}{\partial_t}

\newcommand{\R}{\mathbb{R}}

\newcommand{\T}{\mathbb{T}}
\newcommand{\ve}{\varepsilon}

\newcommand{\Z}{\mathbb{Z}}
\newcommand{\be}{\begin{equation}}
\newcommand{\ee}{\end{equation}}

\newcommand{\la}{\label}
\newcommand{\fr}{\frac}

\DeclareMathOperator{\curl}{\mathrm{curl}}

\title[The Nernst-Planck-Darcy System]{Charged Fluids in Porous Media}
\date{\today}

\author{Mihaela Ignatova}
\address{Department of Mathematics, Temple University, Philadelphia, PA 19122}
\email{ignatova@temple.edu}
\author{Jingyang Shu}
\address{Department of Mathematics, Temple University, Philadelphia, PA 19122}
\email{jyshu@temple.edu}

\begin{document}

\begin{abstract}
The Nernst-Planck-Darcy system models ionic electrodiffusion in porous media. We consider the system for two ionic species with opposite valences. We prove that the initial value problem for the Nernst-Planck-Darcy system in periodic domains in two or three dimensions has global weak solutions in $W^{1, p}$ ($p \geq 2$). We obtain furthermore global existence and uniqueness of smooth solutions for arbitrary large data.
\end{abstract}

\maketitle

\section{Introduction}

In this paper, we consider the Nernst-Planck (NP) system
\begin{align}
& \pt c_i + \nabla \cdot \left(u c_i - D_i \nabla c_i - z_i D_i c_i \nabla \Phi\right) = 0, \qquad i = 1, \dotsc, N, \label{NP}\\
& - \ve \Delta \Phi = \rho, \label{poisson0}\\
& \rho = \sum_{i = 1}^N z_i c_i, \label{charge}
\end{align}
coupled with Darcy's law
\begin{align}
& u + \nabla p = - \rho \nabla \Phi, \label{darcy0}\\
& \nabla \cdot u = 0, \label{incomp0}
\end{align}
where $c_i \colon \T^d \times [0, T] \to \R^+$ are the $i$-th ionic species concentrations, $z_i \in \Z$ are corresponding valences, $D_i > 0$ are constant diffusivities, $\Phi \colon \T^d \times [0, T] \to \R$ is the nondimensional electrical potential, $\rho \colon \T^d \times [0, T] \to \R$ is the nondimensional charge density, $u \colon \T^d \times [0, T] \to \T^d$ is the fluid velocity field, $p \colon \T^d \times [0, T] \to \R$ is the fluid pressure, and $\ve > 0$ is a constant proportional to the square of the Debye length \cite{CI19,Rub90}. Here, $d = 2, 3$ is the space dimension and $\T^d$ is a $d$-dimensional torus. 

We refer to the equations \eqref{NP}--\eqref{incomp0} as the Nernst-Planck-Darcy (NPD) equations. The NPD system models the transport and diffusion of ions by incompressible flow in porous media. The Nernst-Planck equation \eqref{NP} describes the evolution of ions which are carried by an incompressible fluid and interact among themselves via an electric potential and thermal fluctuations. The Poisson-Boltzmann equation \eqref{poisson0}--\eqref{charge} relates the electrostatic potential to the ionic charge density. Darcy's law replaces the viscous dissipation with friction, representing the damping of flow through porous media. We take the friction coefficient to be one, for simplicity. Darcy's law \eqref{darcy0} and the incompressibility condition \eqref{incomp0} describe then the flow of an incompressible fluid through a porous medium forced by the macroscopic electrostatic Lorentz force created by the ions.

Ionic electrodiffusion is important to many fields, including biology, chemistry and physics, and has wide applications \cite{Rub90}. There are extensive mathematical studies of models coupling the Nernst-Planck equations with various fluid dynamical systems. For the fluids that are described by the Navier-Stokes equations, the system is known as the Nernst-Planck-Navier-Stokes (NPNS) system. In the whole space $\R^d$ ($d = 2, 3$), local existence of solutions of the NPNS system is obtained in \cite{Jer02}, and later weak solutions are proved to exist globally in time \cite{LW20}, with some $L^2$ decay if the dimension is two \cite{ZY15b}. In bounded domains $\Omega \subset \R^d$, the NPNS system has global solutions under various appropriate boundary conditions. For example, with blocking boundary conditions, global weak solutions exist in both two and three dimensions \cite{FS17, JS09}. For ionic concentrations satisfying blocking boundary conditions while the electrical potential satisfying the Dirichlet boundary condition, global weak solutions exist in two and three dimensions if the initial data is small \cite{Ryh09p}. In the case of blocking boundary conditions for the ionic concentrations and homogeneous Neumann boundary condition for the electrical potential, weak solutions are global in two dimensions \cite{Sch09}. Moreover, with blocking boundary conditions for the ionic concentrations and Robin boundary condition for the electrical potential, two dimensional strong solutions are global \cite{BFS14}. The same result holds in three dimensions if the fluid velocity remains regular for all time \cite{Lee21p}. When the ionic concentrations satisfy either the blocking boundary conditions or the uniformly selective boundary conditions, strong solutions are global in two dimensions \cite{CI19}, and in three dimensions provided that the initial data is a small perturbation of a steady state \cite{CIL20pa}. If both the ionic concentrations and the electrical potential obey the Dirichlet boundary conditions, global strong solutions exist in three dimensions as long as the fluid velocity is regular \cite{CIL20pb}. With periodic boundary conditions, two dimensional strong solutions exist globally in time, and long time behaviors of the solutions are studied in \cite{AI20p} under the influence of body forces or body charges. Some zero Debye length limit ($\ve \to 0$ in \eqref{poisson0}) results for the NPNS system are proved in \cite{CIL20pc, Li09, WJ21, WJL19}. In the limit of zero viscosity in the Navier-Stokes equations, the solutions of NPNS system in two dimensions converges to the solutions of the corresponding Nernst-Planck-Euler (NPE) system, whose solutions exist and are global \cite{IS21p, ZY15, ZY20}. For the Nernst-Planck system coupled with time dependent Stokes equations, solutions are known existing globally in time in three dimensions \cite{CIL20pb, Lee21p}.

For the NPD system, global existence of weak solutions exist in $L^2$-based Sobolev spaces \cite{HK16pa, HK16pb, HRK12}.

We focus on the initial value problem for the NPD equations \eqref{NP}--\eqref{incomp0} in two or three space dimensions ($d = 2, 3$) with two ionic species ($N = 2$) with opposite valences ($z_1 = - z_2 = 1$) and with equal diffusivities ($D_1 = D_2 = D$). The initial data of the system is
\begin{align}
\label{initc}
& c_i(\cdot , 0) = c_i(0), \quad i = 1, 2,
\end{align}
where the ionic concentrations are nonnegative, $c_i(0) \geq 0$, and the electric charge obeys
\begin{align}
\label{neut}
\int_{\T^d} \rho(x, 0) \diff{x} = \sum_{i = 1}^2 \int_{\T^d} z_i c_i(x, 0) \diff{x} = \int_{\T^d} c_1(x, 0) - c_2(x, 0) \diff{x} = 0.
\end{align}
It follows from \eqref{NP} that the property \eqref{neut} is preserved in time.
For regular solutions of NPNS it is shown in \cite{CI19, CIL20pb} that if $c_i(0) \geq 0$, then $c_i(x, t)$ remains nonnegative for $t > 0$. This property follows from \eqref{NP} if $c_i$ are known to be sufficiently regular, and the same proof and result holds for the NPD equations.

We denote $\rho = c_1 - c_2$, $\sigma = c_1 + c_2$.  Using \eqref{incomp0}, the Nernst-Planck system \eqref{NP} is equivalent to the equations
\begin{align}
\begin{split}
\pt \rho & = - u \cdot \nabla \rho + D (\Delta \rho + \nabla \sigma \cdot \nabla \Phi + \sigma \Delta \Phi),
\end{split} \label{rhot}\\
\begin{split}
\pt \sigma & = - u \cdot \nabla \sigma + D (\Delta \sigma + \nabla \rho \cdot \nabla \Phi + \rho \Delta \Phi).
\end{split} \label{sigmat0}
\end{align}
We have from \eqref{poisson0} that
\begin{align}
\label{poisson}
- \ve \Delta \Phi = \rho,
\end{align}
and from \eqref{darcy0}--\eqref{incomp0} that
\begin{align}
& u + \nabla p = - \rho \nabla \Phi, \label{darcy}\\
& \nabla \cdot u = 0. \label{incomp}
\end{align}
The system \eqref{rhot}--\eqref{incomp} has initial data from \eqref{initc},
\begin{align}
\label{initial}
\begin{split}
& \rho(\cdot, 0) = \rho(0) = c_1(0) - c_2(0),\\
& \sigma (\cdot, 0) = \sigma(0) = c_1(0) + c_2(0),
\end{split}
\end{align}
and 
\[
c_1 = \frac{\sigma + \rho}{2} \quad \text{and} \quad c_2 = \frac{\sigma - \rho}{2}
\]
solve the original Nernst-Planck-Darcy system \eqref{NP}--\eqref{incomp0}.

We fix the parameters $\ve > 0$ and $D > 0$ in \eqref{rhot}--\eqref{poisson} and will not specify them in the rest of the paper.

In this paper, we establish the existence of global weak solutions in $W^{1, r}$ Sobolev spaces and the existence and uniqueness of strong solutions. Both results hold for arbitrary large data. The main theorems are the following.

\begin{theorem}
\label{thm:weakglobal}
Let $d = 2, 3$ and $r \geq 2$. Let $c_1(0), c_2(0) \in W^{1, r}(\T^d)$ be nonnegative functions satisfying \eqref{neut}. Then for any $T > 0$, there exist $c_1(x,t)\ge 0$, $c_2(x,t)\ge 0$ and $u(x,t)$, divergence free, such that $c_1 - c_2 = \rho$ and $c_1 + c_2 = \sigma$ obey $\rho, \sigma \in L^{\infty}([0, T]; W^{1, r}(\T^d)) \cap L^2(0, T; H^2(\T^d))$,  $u(x,t)$ obeys $u \in L^{\infty}([0, T]; W^{1, r}(\T^d))$, and $(\rho, \sigma, u)$ solve the initial value problem \eqref{rhot}--\eqref{initial} in the sense of distributions. The charge density $\rho$ and total concentration $\sigma$ satisfy the following bounds
\begin{align}
\label{solest1}
\begin{split}
& (i) \quad \|\rho(t)\|_{L^p} + \left\|\sigma(t) - \bar{\sigma}\right\|_{L^p} \leq C_p e^{- C' t}, \quad \quad \forall p\ge 2,\\
& (ii) \quad \|\nabla \Phi(t)\|_{L^\infty} \leq C e^{- C' t},\\
& (iii) \quad \|\nabla \rho(t)\|_{L^2}^2 + \|\nabla \sigma(t)\|_{L^2}^2 + \int_0^t \|\Delta \rho(\tau)\|_{L^2}^2 + \|\Delta \sigma(\tau)\|_{L^2}^2 \diff{\tau} \leq C,\\
& (iv) \quad \left\|\nabla \rho(t)\right\|_{L^r} + \left\|\nabla \sigma(t)\right\|_{L^r} \leq C e^{C t},
\end{split}
\end{align}
with  constants $C, C_p > 0$ depending on $D$, $\ve$, $p$, $r$, and the initial data $\|\rho(0)\|_{W^{1, r}}$ and $\|\sigma(0)\|_{W^{1, r}}$ and with $C'>0$ bounded below independently of $p$. Moreover, the fluid velocity satisfies the bound
\begin{align}
\label{solest2}
\|\nabla u(t)\|_{L^r} \leq C e^{C'' t},
\end{align}
where $C'' \in \R$ is a constant depending on $\|\rho(0)\|_{W^{1, r}}$ and $\|\sigma(0)\|_{W^{1, r}}$.
\end{theorem}
The existence of globally smooth solutions of three dimensional Nernst-Planck equations with arbitrary large data coupled to Stokes equations driven by the Lorentz force has been obtained only recently in \cite{CIL20pb, Lee21p}. 
Replacing viscous dissipation by friction diminishes the regularizing effect by two differential orders. In this paper we use the nonlinear structure of the NP equations coupled to Darcy's law, in order to prove the global regularity. The construction of solutions can be achieved by a two-tier approximation procedure as in \cite{CEIV17, IS21p}. The two-tier approximation scheme is needed because the global a priori $L^p$-estimate \eqref{solest1}(i) is a consequence of the special nonlinear structure of \eqref{rhot}--\eqref{sigmat0}, and standard approximation schemes do not preserve the resulting nonlinear dissipation (see \eqref{L2estin}). In this paper we focus on the a~priori estimates, which are the heart of the matter.

\begin{theorem}
\label{thm:strongglobal}
Let $d = 2, 3$. Let $c_1(0), c_2(0) \in H^3(\T^d)$ be nonnegative functions satisfying \eqref{neut}. Then for any $T > 0$, there exists a unique strong solution $\rho, \sigma \in L^{\infty}([0, T]; H^3(\T^d)) \cap L^2(0, T; H^4(\T^d))$ and $u \in L^{\infty}([0, T]; H^3(\T^d))$ of the initial value problem \eqref{rhot}--\eqref{initial}. In addition to the bounds \eqref{solest1} for the ionic concentrations, we also have for any $t > 0$,
\bas{3dHsbdd1}
& \|\Delta \rho(t)\|_{L^2} + \|\Delta \sigma(t)\|_{L^2} + \int_0^t \|\nabla \Delta \rho(\tau)\|_{L^2}^2 + \|\nabla \Delta \sigma(\tau)\|_{L^2}^2 \leq C,\\
& \|\nabla \Delta \rho(t)\|_{L^2} + \|\nabla \Delta \sigma(t)\|_{L^2} + \int_0^t \|\Delta^2 \rho(\tau)\|_{L^2}^2 + \|\Delta^2 \sigma(\tau)\|_{L^2}^2 \leq C,
\eas
where $C > 0$ depends only on $\ve$, $D$, and the initial data. For the fluid velocity $u$, in addition to the estimates \eqref{solest2}, we also have for any $t > 0$
\bax
\|u(t)\|_{H^3} \leq C.
\eax
\end{theorem}

The proof of Theorem~\ref{thm:weakglobal} is in Section~\ref{sec:weak} and the proof of Theorem~\ref{thm:strongglobal} is in Section~\ref{sec:strong}. In the sequel we omit the integration domain $\T^d$ and write $\int f = \int_{\T^d} f(x) \diff{x}$. We denote 
\[
\bar{f} = \frac{1}{|\T^d|} \int f,
\]
the average of a function $f$ over the torus $\T^d$. In inequalities, $C$ and $C'$ denote constants which may change from line to line.

\section{Global existence of weak solutions}
\label{sec:weak}

In this section, we present the a priori estimates in Theorem~\ref{thm:weakglobal}.  We split the proof of the a priori estimates into several lemmas.

The first lemma concerns with the zeroth energy balance of the system \eqref{rhot}--\eqref{incomp}.
\begin{lemma}
Let $d = 2, 3$ and $r \geq 2$. Let $c_1(0), c_2(0) \in W^{1, r}(\T^d)$ be nonnegative functions satisfying \eqref{neut}. Suppose $(\rho, \sigma, u)$ solves \eqref{rhot}--\eqref{incomp} with initial data \eqref{initial} on the interval $[0, T]$. Then for any $t \in [0, T]$, we have
\bas{nablaPhi-02}
\|\nabla \Phi(t)\|_{L^2}^2 + \frac{2}{\ve} \int_0^t \|u(\tau)\|_{L^2}^2 \diff{\tau} + \frac{2 D}{\ve} \int_0^t \|\rho(\tau)\|_{L^2}^2 \diff{\tau} \leq 2 \|\nabla \Phi(0)\|_{L^2}^2.
\eas
\end{lemma}

\begin{proof}
Using \eqref{rhot} and \eqref{poisson}, we have an evolutionary equation
\[
\pt (- \Delta \Phi) = - \frac{1}{\ve} u \cdot \nabla \rho + \frac{D}{\ve} (\Delta \rho + \nabla \sigma \cdot \nabla \Phi + \sigma \Delta \Phi).
\]
Testing this equation with $\Phi$ and integrating by parts, we find
\bax
\frac{1}{2} \frac{\diff}{\diff{t}} \|\nabla \Phi\|_{L^2}^2 & = - \frac{1}{\ve} \int u \cdot \nabla \rho \Phi + \frac{D}{\ve} \int \Delta \rho \Phi + \frac{D}{\ve} \int \nabla \cdot (\sigma \nabla \Phi) \Phi\\
& = \frac{1}{\ve} \int u \cdot (\rho \nabla \Phi) + \frac{D}{\ve} \int \rho \Delta \Phi - \frac{D}{\ve} \int \sigma |\nabla \Phi|^2.
\eax
By equations \eqref{darcy} and \eqref{poisson}, we obtain
\bax
\frac{1}{2} \frac{\diff}{\diff{t}} \|\nabla \Phi\|_{L^2}^2 + \frac{1}{\ve} \|u\|_{L^2}^2 + \frac{D}{\ve^2} \|\rho\|_{L^2}^2 + \frac{D}{\ve} \int \sigma |\nabla \Phi|^2 = 0.
\eax
Integrating in time and discarding the last term on the left hand side (notice that $\sigma = c_1 + c_2 \geq 0$) give the inequality \eqref{nablaPhi-02}.
\end{proof}

The following lemma gives some useful bounds for the fluid velocity field $u$.
\begin{lemma}
Let $d = 2, 3$. Let $(\rho, u)$ satisfy \eqref{darcy}--\eqref{incomp}. Then
\bal
& \|u\|_{L^p} \leq C \|\rho \nabla \Phi\|_{L^p} \leq C \|\rho\|_{L^p} \|\nabla \Phi\|_{L^\infty}, \label{uLp}\\
& \|\nabla u\|_{L^r} \leq C \|\nabla \rho\|_{L^r} \|\nabla \Phi\|_{L^\infty}, \label{nabuLr}
\eal
for all $p, r \in (1, \infty)$.
\end{lemma}
\begin{proof}
Applying the Leray projector to \eqref{darcy} and noting that the Leray projector is bounded on $L^p$ when $p \in (1, \infty)$ (see, e.g., \cite{CF88}), we obtain \eqref{uLp}. As for \eqref{nabuLr}, we take the curl of \eqref{darcy} and have
\bax
& \nabla^\perp \cdot u = - \nabla^\perp \rho \cdot \nabla \Phi \quad \text{if}\ d = 2,\\
& \curl u = - \nabla \rho \times \nabla \Phi \quad \text{if}\ d = 3.
\eax
Then the proof is completed by invoking the well-known estimate
\[
\|\nabla u\|_{L^r} \leq C \|\curl u\|_{L^r}
\]
for all $r \in (1, \infty)$.
\end{proof}

A key step of proving global a priori bounds for the weak solutions is to obtain \eqref{solest1}(ii) and thus \eqref{solest1}(i). In either two or three dimensions, in view of \eqref{poisson}, elliptic estimates, and Sobolev embeddings, we have
\bal
& \|\nabla \Phi\|_{L^6(\T^d)} \leq C \|\rho\|_{L^2(\T^d)}, \label{Phiemb1}\\
& \|\nabla \Phi\|_{L^\infty(\T^d)} \leq C \|\rho\|_{L^4(\T^d)}, \label{Phiemb3}
\eal
for $d = 2, 3$.

The following lemma states the pointwise exponential decay of $\|\rho\|_{L^p}$ and $\|\nabla \Phi\|_{L^\infty}$.
\begin{lemma}
\label{lem:rhodecay}
Let $d = 2, 3$ and $r \geq 2$. Let $c_1(0), c_2(0) \in W^{1, r}(\T^d)$ be nonnegative functions satisfying \eqref{neut}. Suppose $(\rho, \sigma, u)$ solves \eqref{rhot}--\eqref{incomp} with initial data \eqref{initial} on the interval $[0, T]$. Then for any $t \in [0, T]$, we have
\bal
& \|\rho(t)\|_{L^p} \leq C_p e^{- C' t}, \quad \quad \forall p\ge 2, \label{rholpdecay}\\
& \|\nabla \Phi(t)\|_{L^\infty} \leq C e^{- C' t}, \label{phidecay2d}
\eal
for some constants $C_p, C' > 0$, with $C'$ independent of $p$.
\end{lemma}
\begin{proof}
We first observe that \eqref{sigmat0} is equivalent to
\begin{align}
\label{sigmat}
\pt (\sigma - \bar{\sigma}) = - u \cdot \nabla (\sigma - \bar{\sigma}) + D \left(\Delta (\sigma - \bar{\sigma}) + \nabla \rho \cdot \nabla \Phi + \rho \Delta \Phi\right).
\end{align}
Note that the average $\bar{\sigma} \geq 0$ since $c_1, c_2 \geq 0$ and $\bar{\sigma}$ is conserved in time due to \eqref{sigmat0}.
Let $p \ge 2$. We multiply \eqref{rhot} by $\frac{1}{p - 1} \rho |\rho|^{p - 2}$ and \eqref{sigmat} by $\frac{1}{p - 1} (\sigma - \bar{\sigma}) |\sigma - \bar{\sigma}|^{p - 2}$, and then integrate by parts,
\begin{align}
\begin{split}
& \frac{1}{p (p - 1)}\frac{\diff}{\diff{t}} \|\rho\|_{L^p}^p = - D \int |\rho|^{p - 2} |\nabla \rho|^2 - D \int |\rho|^{p - 2} (\sigma - \bar{\sigma}) \nabla \rho \cdot \nabla \Phi - D \int |\rho|^{p - 2} \bar{\sigma} \nabla \rho \cdot \nabla \Phi,
\end{split} \label{rholpinq}\\
& \frac{1}{p (p - 1)} \frac{\diff}{\diff{t}} \|\sigma - \bar{\sigma}\|_{L^p}^p = - D \int |\sigma - \bar{\sigma}|^{p - 2} |\nabla (\sigma - \bar{\sigma})|^2 - D \int |\sigma - \bar{\sigma}|^{p - 2} \rho \nabla(\sigma - \bar{\sigma}) \cdot \nabla \Phi. \label{sigmalpinq}
\end{align}
Taking $p = 2$, summing \eqref{rholpinq} and \eqref{sigmalpinq}, and using \eqref{poisson}, we obtain
\bax
\frac{1}{2} \frac{\diff}{\diff{t}} \left(\|\rho\|_{L^2}^2 + \|\sigma - \bar{\sigma}\|_{L^2}^2\right) + D \left(\|\nabla \rho\|_{L^2}^2 + \|\nabla (\sigma - \bar{\sigma})\|_{L^2}^2\right) + \frac{D}{\ve} \int \sigma \rho^2 = 0.
\eax
Recall that the ionic concentrations $c_1, c_2 \geq 0$, so that $\sigma = c_1 + c_2 \geq 0$, and thus, the last term on the left hand side is nonnegative. Furthermore, since $|\rho| = |c_1 - c_2| \leq c_1 + c_2 = \sigma$, we have that
\begin{align}
\label{L2estin}
\frac{1}{2} \frac{\diff}{\diff{t}} \left(\|\rho\|_{L^2}^2 + \|\sigma - \bar{\sigma}\|_{L^2}^2\right) + D \left(\|\nabla \rho\|_{L^2}^2 + \|\nabla (\sigma - \bar{\sigma})\|_{L^2}^2\right) + \frac{D}{\ve} \|\rho\|_{L^3}^3 \leq 0.
\end{align}
By the Poincar\'{e} inequality and Gr\"{o}nwall's inequality, we deduce the following exponential pointwise decay
\begin{align}
\label{l2decay}
\|\rho(t)\|_{L^2}^2 + \|\sigma(t) - \bar{\sigma}\|_{L^2}^2 \leq \left(\|\rho(0)\|_{L^2}^2 + \|\sigma(0) - \bar{\sigma}\|_{L^2}^2\right) e^{- C' t},
\end{align}
and the bounds
\begin{align}
\label{l2time}
2 D \int_0^t \|\nabla \rho(\tau)\|_{L^2}^2 + \|\nabla \sigma(\tau)\|_{L^2}^2 \diff{\tau} + \frac{2 D}{\ve} \int_0^t \|\rho(\tau)\|_{L^3}^3 \diff{\tau} \leq \|\rho(0)\|_{L^2}^2 + \|\sigma(0) - \bar{\sigma}\|_{L^2}^2.
\end{align}

We obtain from \eqref{Phiemb1} and \eqref{l2decay} that
\bas{phidecay}
\|\nabla \Phi(t)\|_{L^6} \leq C \|\rho(t)\|_{L^2} \leq C e^{- C' t}.
\eas

For $p \geq 4$, we have from \eqref{rhot}--\eqref{poisson} that
\bal
& \frac{1}{p (p - 1)} \frac{\diff}{\diff{t}} \|\rho\|_{L^p}^p = - D \int |\rho|^{p - 2} |\nabla \rho|^2 - \frac{D \bar{\sigma}}{(p - 1) \ve} \int |\rho|^p - D \int |\rho|^{p - 2} (\sigma - \bar{\sigma}) \nabla \rho \cdot \nabla \Phi, \label{dtrhoLp}\\
& \frac{1}{p (p - 1)} \frac{\diff}{\diff{t}} \|\sigma - \bar{\sigma}\|_{L^p}^p = - D \int |\sigma - \bar{\sigma}|^{p - 2} |\nabla \sigma|^2 - D \int |\sigma - \bar{\sigma}|^{p - 2} \rho \nabla \sigma \cdot \nabla \Phi. \label{dtsigmaLp}
\eal

Notice that by H\"{o}lder's inequality, the Gagliardo--Nirenberg interpolation inequality
\[
\|f\|_{L^\frac{3 (p - 2)}{p - 3}} \leq C \|\nabla f\|_{L^2}^{\frac{d p}{6 (p - 2)}} \|f\|_{L^2}^{1 - \frac{d p}{6 (p - 2)}} + \|f\|_{L^2},
\]
and Young's inequality, we have
\bas{dtrhoLpest1}
& D \int |\rho|^{p - 2} (\sigma - \bar{\sigma}) \nabla \rho \cdot \nabla \Phi\\
& \quad \leq \frac{2 D}{p} \|\nabla \Phi\|_{L^6} \|\nabla |\rho|^{\frac{p}{2}}\|_{L^2} \||\rho|^{\frac{p - 2}{2}} (\sigma - \bar{\sigma})\|_{L^3}\\
& \quad \leq \frac{2 D}{p} \|\nabla \Phi\|_{L^6} \|\nabla |\rho|^{\frac{p}{2}}\|_{L^2} \||\rho|^{\frac{p}{2}}\|_{L^{\frac{3 (p - 2)}{p - 3}}}^{\frac{p - 2}{p}} \|\sigma - \bar{\sigma}\|_{L^p}\\
& \quad \leq C \|\nabla \Phi\|_{L^6} \|\nabla |\rho|^{\frac{p}{2}}\|_{L^2}^{1 + \frac{d}{6}} \|\rho\|_{L^p}^{\frac{(6 - d) p}{12} - 1} \|\sigma - \bar{\sigma}\|_{L^p} + C \|\nabla \Phi\|_{L^6} \|\nabla |\rho|^{\frac{p}{2}}\|_{L^2} \|\rho\|_{L^p}^{\frac{p - 2}{2}} \|\sigma - \bar{\sigma}\|_{L^p}\\
& \quad \leq \frac{D}{2} \int |\rho|^{p - 2} |\nabla \rho|^2 + C \|\nabla \Phi\|_{L^6}^{\frac{12}{6 - d}} \|\rho\|_{L^p}^{p - \frac{12}{6 - d}} \|\sigma - \bar{\sigma}\|_{L^p}^{\frac{12}{6 - d}} + C \|\nabla \Phi\|_{L^6}^2 \|\rho\|_{L^p}^{p - 2} \|\sigma - \bar{\sigma}\|_{L^p}^2.
\eas
Similarly, we have
\bas{dtsigmaLpest1}
& D \int |\sigma - \bar{\sigma}|^{p - 2} \rho \nabla \sigma \cdot \nabla \Phi\\
& \quad \leq \frac{D}{2} \int |\sigma - \bar{\sigma}|^{p - 2} |\nabla \sigma|^2 + C \|\nabla \Phi\|_{L^6}^{\frac{12}{6 - d}} \|\sigma - \bar{\sigma}\|_{L^p}^{p - \frac{12}{6 - d}} \|\rho\|_{L^p}^{\frac{12}{6 - d}} + C \|\nabla \Phi\|_{L^6}^2 \|\sigma - \bar{\sigma}\|_{L^p}^{p - 2} \|\rho\|_{L^p}^2.
\eas
Using \eqref{dtrhoLpest1}--\eqref{dtsigmaLpest1} into \eqref{dtrhoLp}--\eqref{dtsigmaLp} yields
\bax
\frac{\diff}{\diff{t}} \left(\|\rho\|_{L^p} + \|\sigma - \bar{\sigma}\|_{L^p}\right) + \frac{D \bar{\sigma}}{\ve} \|\rho\|_{L^p} \leq C \left(\|\nabla \Phi\|_{L^6}^2 + \|\nabla \Phi\|_{L^6}^{\frac{12}{6 - d}}\right) \left(\|\rho\|_{L^p} + \|\sigma - \bar{\sigma}\|_{L^p}\right).
\eax
Dropping the dissipation term and applying Gr\"{o}nwall's inequality and \eqref{phidecay} give
\bas{rhosimgaLpconstbdd}
\|\rho(t)\|_{L^p} + \|\sigma(t) - \bar{\sigma}\|_{L^p} \leq \left(\|\rho(0)\|_{L^p} + \|\sigma(0) - \bar{\sigma}\|_{L^p}\right) e^{C \int_0^t \|\nabla \Phi(\tau)\|_{L^6}^2 + \|\nabla \Phi(\tau)\|_{L^6}^{\frac{12}{6 - d}} \diff{\tau}} \leq C,
\eas
where $C > 0$ is a constant depending only on $p$, the initial data, and the parameters of the problem.

Now we use \eqref{dtrhoLpest1} again in \eqref{dtrhoLp}, and update with the new estimate \eqref{rhosimgaLpconstbdd} to derive that
\bax
\frac{1}{p (p - 1)} \frac{\diff}{\diff{t}} \|\rho\|_{L^p}^p + \frac{D \bar{\sigma}}{(p - 1) \ve} \|\rho\|_{L^p}^p & \leq C \|\nabla \Phi\|_{L^6}^{\frac{12}{6 - d}} \|\rho\|_{L^p}^{p - \frac{12}{6 - d}} \|\sigma - \bar{\sigma}\|_{L^p}^{\frac{12}{6 - d}} + C \|\nabla \Phi\|_{L^6}^2 \|\rho\|_{L^p}^{p - 2} \|\sigma - \bar{\sigma}\|_{L^p}^2\\
& \leq C \left(\|\nabla \Phi\|_{L^6}^2 + \|\nabla \Phi\|_{L^6}^{\frac{12}{6 - d}}\right).
\eax
Integrating in time and applying \eqref{phidecay} then yield
\bax
\|\rho(t)\|_{L^p}^p \leq e^{- \frac{D \bar{\sigma}p}{\ve} t} \bigg(\|\rho(0)\|_{L^p}^p + C \int_0^t \|\nabla \Phi(\tau)\|_{L^6}^2 + \|\nabla \Phi(\tau)\|_{L^6}^{\frac{12}{6 - d}} \diff{\tau}\bigg) \leq C e^{- \frac{D \bar{\sigma}p}{\ve} t},
\eax
which leads to \eqref{rholpdecay}. Finally, in view of \eqref{Phiemb3} and \eqref{rholpdecay}, we conclude \eqref{phidecay2d}.
\end{proof}

The following lemma establishes the pointwise decay of $\|\sigma - \bar{\sigma}\|_{L^p}$, whose proof is based on a Moser's type iteration argument as in \cite{BFS14, CL95, CIL20pa, IS21p}.
\begin{lemma}
Let $d = 2, 3$ and $r \geq 2$. Let $c_1(0), c_2(0) \in W^{1, r}(\T^d)$ be nonnegative functions satisfying \eqref{neut}. Suppose $(\rho, \sigma, u)$ solves \eqref{rhot}--\eqref{incomp} with initial data \eqref{initial} on the interval $[0, T]$. Then for any $t \in [0, T]$, we have
\bas{sigmalpdecay}
& \|\sigma(t) - \bar{\sigma}\|_{L^p} \leq C_p e^{- C' t}, \quad \quad \forall p\ge 2,
\eas
for some constants $C_p, C' > 0$ with $C'$ independent of $p$.
\end{lemma}
\begin{proof}
From \eqref{sigmalpinq}, we have
\begin{align*}
\frac{1}{p} \frac{\diff}{\diff{t}} \left\||\sigma - \bar{\sigma}|^{\frac{p}{2}}\right\|_{L^2}^2 + D (p - 1) \int |\sigma - \bar{\sigma}|^{p - 2} |\nabla (\sigma - \bar{\sigma})|^2 = - D (p - 1) \int |\sigma - \bar{\sigma}|^{p - 2} \rho \nabla (\sigma - \bar{\sigma}) \cdot \nabla \Phi.
\end{align*}
We use the bounds
\[
D (p - 1) \int |\sigma - \bar{\sigma}|^{p - 2} |\nabla (\sigma - \bar{\sigma})|^2 = \frac{4 D (p - 1)}{p^2} \int \left|\nabla |\sigma - \bar{\sigma}|^{\frac{p}{2}}\right|^2 \geq \frac{2 D}{p} \int \left|\nabla |\sigma - \bar{\sigma}|^{\frac{p}{2}}\right|^2
\]
and
\[
D  (p - 1) \int |\sigma - \bar{\sigma}|^{p - 2} \rho \nabla (\sigma - \bar{\sigma}) \cdot \nabla \Phi \leq 2 D \|\rho\|_{L^p} \|\nabla \Phi\|_{L^\infty} \left\||\sigma - \bar{\sigma}|^{\frac{p}{2}}\right\|_{L^2}^{\frac{p - 2}{p}} \left\|\nabla |\sigma - \bar{\sigma}|^{\frac{p}{2}}\right\|_{L^2}
\]
to deduce 
\[
\frac{\diff}{\diff{t}} \left\||\sigma - \bar{\sigma}|^{\frac{p}{2}}\right\|_{L^2}^2 + 2 D \left\|\nabla |\sigma - \bar{\sigma}|^{\frac{p}{2}}\right\|_{L^2}^2 \leq 2 D p \|\rho\|_{L^p} \|\nabla \Phi\|_{L^\infty} \left\||\sigma - \bar{\sigma}|^{\frac{p}{2}}\right\|_{L^2}^{\frac{p - 2}{p}} \left\|\nabla |\sigma - \bar{\sigma}|^{\frac{p}{2}}\right\|_{L^2}.
\]
By Young's inequality, we have
\begin{align}
\label{sigmaest0}
\frac{\diff}{\diff{t}} \left\||\sigma - \bar{\sigma}|^{\frac{p}{2}}\right\|_{L^2}^2 + D \left\|\nabla |\sigma - \bar{\sigma}|^{\frac{p}{2}}\right\|_{L^2}^2 \leq D p^2 \|\rho\|_{L^p}^2 \|\nabla \Phi\|_{L^\infty}^2 \left\||\sigma - \bar{\sigma}|^{\frac{p}{2}}\right\|_{L^2}^{2 - \frac{4}{p}}.
\end{align}
The Gagliardo-Nirenberg interpolation inequality and Young's inequality imply that
\begin{align}
\left\||\sigma - \bar{\sigma}|^{\frac{p}{2}}\right\|_{L^2}^2 & \leq M \left\|\nabla |\sigma - \bar{\sigma}|^{\frac{p}{2}}\right\|_{L^2}^{\frac{2 d}{2 + d}} \left\||\sigma - \bar{\sigma}|^{\frac{p}{2}}\right\|_{L^1}^{\frac{4}{2 + d}} + M \left\||\sigma - \bar{\sigma}|^{\frac{p}{2}}\right\|_{L^1}^2\notag\\
& \leq \delta \left\|\nabla |\sigma - \bar{\sigma}|^{\frac{p}{2}}\right\|_{L^2}^2 + \frac{2 d^{\frac{d}{2}}}{(2 + d)^{1 + \frac{d}{2}}} \frac{M^{\frac{2 + d}{2}} +  M \delta^{\frac{d}{2}}}{\delta^{\frac{d}{2}}} \left\||\sigma - \bar{\sigma}|^{\frac{p}{2}}\right\|_{L^1}^2,
\label{secondtermest}
\end{align}
where $M > 0$ is the constant from the interpolation inequality and $\delta$ is a number to be chosen later.

Multiplying \eqref{secondtermest} by $\frac{D}{\delta}$, we get 
\begin{align}
\label{secondtermest2}
D \left\|\nabla |\sigma - \bar{\sigma}|^{\frac{p}{2}}\right\|_{L^2}^2  \geq \frac{D}{\delta} \left\||\sigma - \bar{\sigma}|^{\frac{p}{2}}\right\|_{L^2}^2 - \frac{2 D d^{\frac{d}{2}}}{(2 + d)^{1 + \frac{d}{2}}} \frac{M^{\frac{2 + d}{2}} +  M \delta^{\frac{d}{2}}}{\delta^{1 + \frac{d}{2}}} \left\||\sigma - \bar{\sigma}|^{\frac{p}{2}}\right\|_{L^1}^2.
\end{align}
Thus, using \eqref{secondtermest2} in \eqref{sigmaest0} yields
\begin{align}
\label{beforeabsorb}
& \frac{\diff}{\diff{t}} \left\||\sigma - \bar{\sigma}|^{\frac{p}{2}}\right\|_{L^2}^2 + \frac{D}{\delta} \left\||\sigma - \bar{\sigma}|^{\frac{p}{2}}\right\|_{L^2}^2\notag\\
&\leq D p^2 \|\rho\|_{L^p}^2 \|\nabla \Phi\|_{L^\infty}^2 \left\||\sigma - \bar{\sigma}|^{\frac{p}{2}}\right\|_{L^2}^{2 - \frac{4}{p}} + \frac{2 D d^{\frac{d}{2}}}{(2 + d)^{1 + \frac{d}{2}}} \frac{M^{\frac{2 + d}{2}} +  M \delta^{\frac{d}{2}}}{\delta^{1 + \frac{d}{2}}} \left\||\sigma - \bar{\sigma}|^{\frac{p}{2}}\right\|_{L^1}^2.
\end{align}

Now we choose $\delta =\fr{1}{2p}$ and use a Young inequality with exponents
$\fr{p}{p-2}$, $\fr{p}{2}$ to deduce
\be
\frac{\diff}{\diff{t}} \left\|\sigma - \bar{\sigma}\right\|_{L^p}^p + pD \left\|\sigma - \bar{\sigma}\right\|_{L^p}^p \leq C_p\left ( \|\rho\|_{L^p}^2 \|\nabla \Phi\|_{L^\infty}^2\right)^{\fr{p}{2}} + C_p\left\|\sigma - \bar{\sigma}\right\|_{L^\frac{p}{2}}^p.
\la{normdiffe}
\ee
Applying Gr\"{o}nwall's inequality then leads to
\begin{align}
\left\|\sigma(t) - \bar{\sigma}\right\|_{L^p}^p
 &\leq e^{- pD t} \bigg[\left\|\sigma(0) - \bar{\sigma}\right\|_{L^p}^p 
 + C_p\int_0^t e^{pD \tau} \|\rho(\tau)\|_{L^p}^p \|\nabla \Phi(\tau)\|_{L^\infty}^p \diff{\tau}\notag\\
 &\quad\quad\quad\quad
 + C_p\int_0^t e^{pD \tau} \left\|\sigma(\tau) - \bar{\sigma}\right\|_{L^{\frac{p}{2}}}^p \diff{\tau}\bigg].
\label{sigmagronwall}
\end{align}
From \eqref{rholpdecay}--\eqref{phidecay2d} it follows that
\begin{align*}
&  \int_0^t e^{-pD (t-\tau)} \|\rho(\tau)\|_{L^p}^p \|\nabla \Phi(\tau)\|_{L^\infty}^p \diff{\tau}
\leq C e^{- pC' t}.
\end{align*}
holds with $C'$ bounded below independently of $p$. 
We estimate the last integral in \eqref{sigmagronwall} by induction. We first recall that $\|\sigma(t) - \bar{\sigma}\|_{L^2}$ decays exponentially in time (see \eqref{l2decay}).

We take $p = 2^{j + 1}$ for $j \in \N$ and assume by induction that
\be
\left\|\sigma(t) - \bar{\sigma}\right\|_{L^p}^p \le C e^{-pc_p t}
\la{induction}
\ee
with $c_p >\epsilon>0$ bounded below independently of $p$. We take, without loss of generality $\epsilon \le \fr{D}{2}$. We deduce from \eqref{sigmagronwall} that
\be
\left\|\sigma(t) - \bar{\sigma}\right\|_{L^p}^p \le Ce^{-pC't} + C te^{-pkt}
\la{indstep}
\ee
with $k = \min\{ D; 2 c_{\fr{p}{2}}\} \ge \min\{D; 2\epsilon\} = 2\epsilon$. Thus $c_p$ is bounded below by $\epsilon>0$ which is uniform in $p\to\infty$.


Therefore, we deduced from \eqref{sigmagronwall} that $\|\sigma(t) - \bar{\sigma}\|_{L^p}$ decays exponentially for each fixed $p \geq 2$ of the form $p = 2^j$ ($j \in \N$) at a rate bounded below uniformly as $p\to\infty$. Then by interpolation, we obtain that $\|\sigma(t) - \bar{\sigma}\|_{L^p}$ decays exponentially for all $p \geq 2$, at a rate bounded below uniformly as $p\to\infty$, which is \eqref{sigmalpdecay}.
\end{proof}

The next lemma concerns with the $W^{1, r}$ norms of the solutions, which finishes the a priori estimates in Theorem~\ref{thm:weakglobal}.
\begin{lemma}
\label{lem:W1r}
Let $d = 2, 3$ and $r \geq 2$. Let $c_1(0), c_2(0) \in W^{1, r}(\T^d)$ be nonnegative functions satisfying \eqref{neut}. Suppose $(\rho, \sigma, u)$ solves \eqref{rhot}--\eqref{incomp} with initial data \eqref{initial} on the interval $[0, T]$. Then for any $t \in [0, T]$, we have
\bal
& \|\nabla \rho(t)\|_{L^2}^2 + \|\nabla \sigma(t)\|_{L^2}^2 + \int_0^t \|\Delta \rho(\tau)\|_{L^2}^2 + \|\Delta \sigma(\tau)\|_{L^2}^2 \diff{\tau} \leq C,\\
& \left\|\nabla \rho(t)\right\|_{L^r} + \left\|\nabla \sigma(t)\right\|_{L^r} \leq C e^{C t}, \label{nabionLr}\\
& \|\nabla u(t)\|_{L^r} \leq C e^{C'' t}, \label{nabuLrbdd}
\eal
for some constants $C > 0$ and $C'' \in \R$.
\end{lemma}
\begin{proof}
Testing \eqref{rhot} and \eqref{sigmat0} with $- \Delta \rho$ and $- \Delta \sigma$ respectively, summing the resulting equations, and using \eqref{poisson}, we obtain
\bas{3dnablaL2}
& \frac{1}{2} \frac{\diff}{\diff{t}} \left(\|\nabla \rho\|_{L^2}^2 + \|\nabla \sigma\|_{L^2}^2\right) + D \left(\|\Delta \rho\|_{L^2}^2 + \|\Delta \sigma\|_{L^2}^2\right) + \frac{D}{\ve} \int \sigma |\nabla \rho|^2\\
& \quad = \int u \cdot \nabla \rho \Delta \rho + \int u \cdot \nabla \sigma \Delta \sigma - D \int \Delta \rho (\nabla \sigma \cdot \nabla \Phi) - D \int \Delta \sigma (\nabla \rho \cdot \nabla \Phi)\\
& \qquad \qquad - \frac{D}{\ve} \int \rho \nabla \rho \cdot \nabla \sigma - \frac{2 D}{\ve} \int \rho |\nabla \rho|^2.
\eas
Using H\"{o}lder's inequality, the advection terms in \eqref{3dnablaL2} can be estimated as
\bas{3dnablaL2-u}
\int u \cdot \nabla \rho \Delta \rho + \int u \cdot \nabla \sigma \Delta \sigma \leq \|u\|_{L^6} \left(\|\nabla \rho\|_{L^3} \|\Delta \rho\|_{L^2} + \|\nabla \sigma\|_{L^3} \|\Delta \sigma\|_{L^2}\right).
\eas
We use the Gagliardo-Nirenberg inequality
\[
\|\nabla f\|_{L^3} \leq C \|\Delta f\|_{L^2}^{\frac{d}{6}} \|\nabla f\|_{L^2}^{1 - \frac{d}{6}} + C \|\nabla f\|_{L^2}
\]
the estimate \eqref{uLp}, and Young's inequality in \eqref{3dnablaL2-u} to obtain
\bax
& \int u \cdot \nabla \rho \Delta \rho + \int u \cdot \nabla \sigma \Delta \sigma\\
& \quad \leq C \|u\|_{L^6} \left(\|\nabla \rho\|_{L^2}^{1 - \frac{d}{6}} \|\Delta \rho\|_{L^2}^{1 + \frac{d}{6}} + \|\nabla \sigma\|_{L^2}^{1 - \frac{d}{6}} \|\Delta \sigma\|_{L^2}^{1 + \frac{d}{6}}\right) + C \|u\|_{L^6} \left(\|\nabla \rho\|_{L^2} \|\Delta \rho\|_{L^2} + \|\nabla \sigma\|_{L^2} \|\Delta \sigma\|_{L^2}\right)\\
& \quad \leq \frac{D}{4} \left(\|\Delta \rho\|_{L^2}^2 + \|\Delta \sigma\|_{L^2}^2\right) + C \left(\|\rho\|_{L^6}^2 \|\nabla \Phi\|_{L^\infty}^2 + \|\rho\|_{L^6}^{\frac{12}{6 - d}} \|\nabla \Phi\|_{L^\infty}^{\frac{12}{6 - d}}\right) \left(\|\nabla \rho\|_{L^2}^2 + \|\nabla \sigma\|_{L^2}^2\right).
\eax

The other terms in \eqref{3dnablaL2} can be estimated using H\"{o}lder's inequality, Young's inequality, and Ladyzhenskaya's inequalities in two or three dimensions. The resulting estimates for \eqref{3dnablaL2} is
\bax
& \frac{\diff}{\diff{t}} \left(\|\nabla \rho\|_{L^2}^2 + \|\nabla \sigma\|_{L^2}^2\right) + 2 D \left(\|\Delta \rho\|_{L^2}^2 + \|\Delta \sigma\|_{L^2}^2\right) + \frac{2 D}{\ve} \int \sigma |\nabla \rho|^2\\
& \quad \leq \frac{D}{2} \left(\|\Delta \rho\|_{L^2}^2 + \|\Delta \sigma\|_{L^2}^2\right) + C \left(\|\rho\|_{L^6}^2 \|\nabla \Phi\|_{L^\infty}^2 + \|\rho\|_{L^6}^{\frac{12}{6 - d}} \|\nabla \Phi\|_{L^\infty}^{\frac{12}{6 - d}}\right) \left(\|\nabla \rho\|_{L^2}^2 + \|\nabla \sigma\|_{L^2}^2\right)\\
& \qquad + C \|\nabla \Phi\|_{L^\infty}^2 \left(\|\nabla \rho\|_{L^2}^2 + \|\nabla \sigma\|_{L^2}^2\right) + C \|\rho\|_{L^2} \left(\|\nabla \rho\|_{L^4}^2 + \|\nabla \sigma\|_{L^4}^2\right)\\
& \quad \leq D \left(\|\Delta \rho\|_{L^2}^2 + \|\Delta \sigma\|_{L^2}^2\right)\\
& \qquad + C \left(\|\nabla \Phi\|_{L^\infty}^2 + \|\rho\|_{L^6}^2 \|\nabla \Phi\|_{L^\infty}^2 + \|\rho\|_{L^6}^{\frac{12}{6 - d}} \|\nabla \Phi\|_{L^\infty}^{\frac{12}{6 - d}} + \|\rho\|_{L^2} + \|\rho\|_{L^2}^{\frac{4}{4 - d}}\right) \left(\|\nabla \rho\|_{L^2}^2 + \|\nabla \sigma\|_{L^2}^2\right).
\eax
Integrating this inequality in time, we obtain
\bax
& \|\nabla \rho(t)\|_{L^2}^2 + \|\nabla \sigma(t)\|_{L^2}^2 + D \int_0^t \|\Delta \rho(\tau)\|_{L^2}^2 + \|\Delta \sigma(\tau)\|_{L^2}^2 \diff{\tau}\\
& \quad \leq \|\nabla \rho(0)\|_{L^2}^2 + \|\nabla \sigma(0)\|_{L^2}^2\\
& \qquad + C \sup_{\tau \in [0, t]} \left(\|\nabla \Phi(\tau)\|_{L^\infty}^2 + \|\rho(\tau)\|_{L^6}^2 \|\nabla \Phi(\tau)\|_{L^\infty}^2 + \|\rho(\tau)\|_{L^6}^{\frac{12}{6 - d}} \|\nabla \Phi(\tau)\|_{L^\infty}^{\frac{12}{6 - d}} + \|\rho(\tau)\|_{L^2} + \|\rho(\tau)\|_{L^2}^{\frac{4}{4 - d}}\right)\\
& \qquad \cdot \int_0^t \|\nabla \rho(\tau)\|_{L^2}^2 + \|\nabla \sigma(\tau)\|_{L^2}^2 \diff{\tau}\\
& \quad \leq C,
\eax
where the last line follows from \eqref{rholpdecay}--\eqref{phidecay2d} and \eqref{l2time}.

For $r > 2$, we differentiate the two equations in \eqref{rhot}--\eqref{sigmat0}, and then take the scalar product with $\nabla \rho |\nabla \rho|^{r - 2}$ and $\nabla \sigma |\nabla \sigma|^{r - 2}$ respectively, integrate over $\T^d$, and integrate by parts to obtain
\bas{3dDrhoLp}
\frac{1}{r} \frac{\diff}{\diff{t}} \|\nabla \rho\|_{L^r}^r & = - \int |\nabla \rho|^{r - 2} \nabla \rho \cdot (\nabla u)^* \nabla \rho - D \int |\nabla \rho|^{r - 2} |\nabla \nabla \rho|^2 - \frac{4 D (r - 2)}{r^2} \int \left|\nabla |\nabla \rho|^{\frac{r}{2}}\right|^2\\
& \qquad - D \int |\nabla \rho|^{r - 2} \Delta \rho \nabla \sigma \cdot \nabla \Phi - D \int |\nabla \rho|^{r - 2} \Delta \rho \sigma \Delta \Phi - D \int \nabla \rho \cdot \nabla |\nabla \rho|^{r - 2} \sigma \Delta \Phi\\
& \qquad  - D (r - 2) \int \nabla \rho \cdot (\nabla \nabla \rho) \cdot \nabla \rho |\nabla \rho|^{r - 4} \nabla \sigma \cdot \nabla \Phi
\eas
and
\bas{3dDsigmaLp}
\frac{1}{r} \frac{\diff}{\diff{t}} \|\nabla \sigma\|_{L^r}^r & = - \int |\nabla \sigma|^{r - 2} \nabla \sigma \cdot (\nabla u)^* \nabla \sigma - D \int |\nabla \sigma|^{r - 2} |\nabla \nabla \sigma|^2 - \frac{4 D (r - 2)}{r^2} \int \left|\nabla |\nabla \sigma|^{\frac{r}{2}}\right|^2\\
& \qquad - D \int |\nabla \sigma|^{r - 2} \Delta \sigma \nabla \rho \cdot \nabla \Phi - D \int |\nabla \sigma|^{r - 2} \Delta \sigma \rho \Delta \Phi - D \int \nabla \sigma \cdot \nabla |\nabla \sigma|^{r - 2} \rho \Delta \Phi\\
& \qquad  - D (r - 2) \int \nabla \sigma \cdot (\nabla \nabla \sigma) \cdot \nabla \sigma |\nabla \sigma|^{r - 4} \nabla \rho \cdot \nabla \Phi.
\eas
For simplicity, we denote
\[
Y = \|\nabla \rho\|_{L^r}^r + \|\nabla \sigma\|_{L^r}^r = \|R\|_{L^2}^2 + \|S\|_{L^2}^2, \quad R = |\nabla \rho|^{\frac{r}{2}}, \quad S = |\nabla \sigma|^{\frac{r}{2}}.
\]
Adding \eqref{3dDrhoLp} to \eqref{3dDsigmaLp} and using \eqref{poisson}, H\"{o}lder's inequality, and Young's inequality, we obtain
\bax
 \frac{\diff}{\diff{t}} Y + \Dissip_1 \leq ~& r \int |\nabla u| (R^2 + S^2) + \frac{D r}{2} \int |\nabla \rho|^{r - 2} |\nabla \nabla \rho|^2 + \frac{D r}{2} \int |\nabla \sigma|^{r - 2} |\nabla \nabla \sigma|^2\\
& + 2 D r ((r - 2)^2 + 1) \|\nabla \Phi\|_{L^\infty}^2 \left(\|\nabla \rho\|_{L^r}^{r - 2} \|\nabla \sigma\|_{L^r}^2 + \|\nabla \sigma\|_{L^r}^{r - 2} \|\nabla \rho\|_{L^r}^2\right)\\
& + \frac{2 D r}{\ve^2} ((r - 2)^2 + 1) \int \rho^2 \left(|\nabla \rho|^{r - 2} \sigma^2 + |\nabla \sigma|^{r - 2} \rho^2\right),
\eax
where $\Dissip_1$ is the dissipation term
\bax
\Dissip_1 = D r \int |\nabla \rho|^{r - 2} |\nabla \nabla \rho|^2  + D r \int |\nabla \sigma|^{r - 2} |\nabla \nabla \sigma|^2 + \frac{4 D (r - 2)}{r} \left(\|\nabla R\|_{L^2}^2 + \|\nabla S\|_{L^2}^2\right).
\eax
Therefore, we have 
\bas{3dYest1}
\frac{\diff}{\diff{t}} Y + \Dissip_2 & \leq r \bigg(\int |\nabla u| (R^2 + S^2)\bigg) + 2 D r ((r - 2)^2 + 1) \|\nabla \Phi\|_{L^\infty}^2 Y\\
& \quad + \frac{2 D r}{\ve^2} ((r - 2)^2 + 1) \int \rho^2 (\rho^2 + \sigma^2) \left(R^{\frac{2 r - 4}{r}} + S^{\frac{2 r - 4}{r}}\right),
\eas
where
\bax
\Dissip_2 & = \frac{D r}{2} \int |\nabla \rho|^{r - 2} |\nabla \nabla \rho|^2  + \frac{D r}{2} \int |\nabla \sigma|^{r - 2} |\nabla \nabla \sigma|^2 + \frac{4 D (r - 2)}{r} \left(\|\nabla R\|_{L^2}^2 + \|\nabla S\|_{L^2}^2\right)\\
& \geq \frac{4 D (r - 2)}{r} \left(\|R\|_{H^1}^2 - \|R\|_{L^2}^2 + \|S\|_{H^1}^2 - \|S\|_{L^2}^2\right).
\eax
We first note that from Ladyzhenskaya's inequality, Young's inequality, and \eqref{nabuLr}
\bas{3dYest1-1}
\int |\nabla u| (R^2 + S^2) & \leq \|\nabla u\|_{L^2} \left(\|R\|_{L^4}^2 + \|S\|_{L^4}^2\right)\\
& \leq \|\nabla u\|_{L^2} \left(\|R\|_{L^2}^{2 - \frac{d}{2}} \|\nabla R\|_{L^2}^{\frac{d}{2}} + \|R\|_{L^2}^2 + \|S\|_{L^2}^{2 - \frac{d}{2}} \|\nabla S\|_{L^2}^{\frac{d}{2}} + \|S\|_{L^2}^2\right)\\
& \leq \frac{D (r - 2)}{r^2} \left(\|R\|_{H^1}^2 + \|S\|_{H^1}^2\right) + C \|\nabla \Phi\|_{L^\infty}^{\frac{4}{4 - d}} \|\nabla \rho\|_{L^2}^{\frac{4}{4 - d}} \left(\|R\|_{L^2}^2 + \|S\|_{L^2}^2\right).
\eas
By H\"{o}lder's inequality with exponents $\fr{r}{2}$ and $\fr{2}{2-\fr{4}{r}}$, we have
\bas{3dYest1-2}
& \frac{2 D r}{\ve^2} ((r - 2)^2 + 1) \int \rho^2 (\rho^2 + \sigma^2) \left(R^{\frac{2 r - 4}{r}} + S^{\frac{2 r - 4}{r}}\right)\\
&\quad\leq C \Big(\|R\|_{L^2}^{2} + \|S\|_{L^2}^{2}\Big)^{\frac{r - 2}{r}} \Big(\|\rho\|_{L^{2r}} + \|\sigma\|_{L^{2r}}\Big)^{4}\\
&\quad\leq \|R\|_{L^2}^2 + \|S\|_{L^2}^2 + C \Big(\|\rho\|_{L^{2r}} + \|\sigma\|_{L^{2r}}\Big)^{2r}.
\eas
Using the inequalities \eqref{3dYest1-1}--\eqref{3dYest1-2} in \eqref{3dYest1}, we get
\bax
\frac{\diff}{\diff{t}} Y \leq C \left(1 + \|\nabla \Phi\|_{L^\infty}^2   + \|\nabla \Phi\|_{L^\infty}^{\frac{4}{4 - d}} \|\nabla \rho\|_{L^2}^{\frac{4}{4 - d}}\right) Y + \Big(\|\rho\|_{L^{2r}} + \|\sigma\|_{L^{2r}}\Big)^{2r}.
\eax
By the bounds \eqref{rholpdecay}--\eqref{phidecay2d},\eqref{sigmalpdecay}, and Gr\"{o}nwall's inequality, we then deduce that $Y(t)$ has at most exponential growth in time $t > 0$,
\bax
&Y(t) =  \|\nabla \rho(t)\|_{L^r}^r + \|\nabla \sigma(t)\|_{L^r}^r\\
&\leq \exp\bigg(C \int_0^t 1 + \|\nabla \Phi(\tau)\|_{L^\infty}^2 + \|\nabla \Phi(\tau)\|_{L^\infty}^{\frac{4}{4 - d}} \|\nabla \rho(\tau)\|_{L^2}^{\frac{4}{4 - d}}\diff{\tau}\bigg)\\
& \qquad \cdot \left[\|\nabla \rho(0)\|_{L^r}^r + \|\nabla \sigma(0)\|_{L^r}^r + \int_0^t\Big(\|\rho(\tau)\|_{L^{2r}} + \|\sigma(\tau)\|_{L^{2r}}\Big)^{2r}\diff{\tau}\right ]\\
&\leq C e^{C' t},
\eax
where the constants $C, C' > 0$ depend on $r$, the parameters of the problem, and the initial data.

We finally use the bounds \eqref{phidecay2d} and \eqref{nabionLr} to obtain that
\begin{align*}
\|\nabla u(t)\|_{L^r} \leq \|\nabla \rho(t)\|_{L^r} \|\nabla \Phi(t)\|_{L^\infty} \leq C e^{C'' t}
\end{align*}
for some constant $C'' \in \R$ depending on $r$, the parameters of the problem, and the initial data.
\end{proof}

\section{Strong solutions of the Nernst-Planck-Darcy system}
\label{sec:strong}

In this section, we present the a priori estimates of Theorem~\ref{thm:strongglobal}. We first note that the embedding $H^3(\T^d) \hookrightarrow W^{1, p}(\T^d)$ (for $p \geq 1$) and Theorem~\ref{thm:weakglobal} imply the global existence of weak solutions together with the bounds \eqref{solest1}--\eqref{solest2} in the interval $[0, T]$ for any $T > 0$. We only need to show the propagation of $H^3$-regularity and the uniqueness of the solutions.  The construction of solutions is similar to the construction of weak solutions.

To prove the estimates in \eqref{3dHsbdd1}, we first prove the lemma 
\begin{lemma}
Let $d = 2, 3$ and $r \geq 2$. Let $c_1(0), c_2(0) \in H^3(\T^d)$ be nonnegative functions satisfying \eqref{neut}. Suppose $(\rho, \sigma, u)$ solves \eqref{rhot}--\eqref{incomp} with initial data \eqref{initial} on the interval $[0, T]$. Then for any $t \in [0, T]$, we have
\bas{deltal2bdd}
&  \|\Delta \rho(t)\|_{L^2} + \|\Delta \sigma(t)\|_{L^2} + \int_0^t \|\nabla \Delta \rho(\tau)\|_{L^2}^2 + \|\nabla \Delta \sigma(\tau)\|_{L^2}^2 \leq C,
\eas
for some constants $C > 0$.
\end{lemma}
\begin{proof}
We multiply \eqref{rhot} and \eqref{sigmat0} by $\Delta^2 \rho$ and $\Delta^2 \sigma$, respectively, and integrate over $\T^d$. Integration-by-parts and \eqref{poisson} give
\begin{align}
\label{DtH2}
\begin{split}
& \frac{\diff}{\diff{t}} \left(\|\Delta \rho\|_{L^2}^2 + \|\Delta \sigma\|_{L^2}^2\right) + D \left(\|\nabla \Delta \rho\|_{L^2}^2 + \|\nabla \Delta \sigma\|_{L^2}^2\right) + \frac{D}{\ve} \int \sigma |\Delta \rho|^2\\
&\quad=  \int \nabla \rho \cdot (\nabla u \nabla \Delta \rho) + \int \nabla \sigma \cdot (\nabla u \nabla \Delta \sigma) + \int u \cdot (\nabla \nabla \rho \nabla \Delta \rho) + \int u \cdot (\nabla \nabla \sigma \nabla \Delta \sigma)\\
&\quad\quad + D \int \nabla \Delta \sigma \cdot \nabla \Phi \Delta \rho + 2 D \int \nabla \nabla \sigma : \nabla \nabla \Phi \Delta \rho - \frac{3 D}{\ve} \int \nabla \sigma \cdot \nabla \rho \Delta \rho\\
&\quad\quad + D \int \nabla \Delta \rho \cdot \nabla \Phi \Delta \sigma + 2 D \int \nabla \nabla \rho : \nabla \nabla \Phi \Delta \sigma - \frac{3 D}{\ve} \int |\nabla \rho|^2 \Delta \sigma - \frac{3 D}{\ve} \int \rho \Delta \sigma \Delta \rho\\
&\quad= I_{1, 1} + I_{1, 2} + I_{1, 3} + I_{1, 4} + I_{1, 5},
\end{split}
\end{align}
where
\begin{align*}
I_{1, 1} & = \int \nabla \rho \cdot (\nabla u \nabla \Delta \rho) + \int \nabla \sigma \cdot (\nabla u \nabla \Delta \sigma),\\
I_{1, 2} & = \int u \cdot (\nabla \nabla \rho \nabla \Delta \rho) + \int u \cdot (\nabla \nabla \sigma \nabla \Delta \sigma),\\
I_{1, 3} & = - \frac{3 D}{\ve} \int \nabla \sigma \cdot \nabla \rho \Delta \rho - \frac{3 D}{\ve} \int |\nabla \rho|^2 \Delta \sigma,\\
I_{1, 4} & = D \int \nabla \Delta \sigma \cdot \nabla \Phi \Delta \rho  + D \int \nabla \Delta \rho \cdot \nabla \Phi \Delta \sigma,\\
I_{1, 5} & = 2 D \int \nabla \nabla \sigma : \nabla \nabla \Phi \Delta \rho + 2 D \int \nabla \nabla \rho : \nabla \nabla \Phi \Delta \sigma - \frac{3 D}{\ve} \int \rho \Delta \sigma \Delta \rho.
\end{align*}

For the term $I_{1, 1}$, we apply H\"{o}lder's inequality, the bound \eqref{nabuLr}, the Gagliardo-Nirenberg inequality
\[
\|\nabla f\|_{L^4} \leq C \|\nabla \Delta f\|_{L^2}^{\frac{d}{8}} \|\nabla f\|_{L^2}^{1 - \frac{d}{8}} + C \|\nabla f\|_{L^2},
\]
and Young's inequality to obtain
\bas{DtH2est1}
I_{1, 1} & \leq \|\nabla u\|_{L^4} \left(\|\nabla \rho\|_{L^4} \|\nabla \Delta \rho\|_{L^2} + \|\nabla \sigma\|_{L^4} \|\nabla \Delta \sigma\|_{L^2}\right)\\
& \leq C \|\nabla \Phi\|_{L^\infty} \|\nabla \rho\|_{L^4} \left(\|\nabla \rho\|_{L^4} \|\nabla \Delta \rho\|_{L^2} + \|\nabla \sigma\|_{L^4} \|\nabla \Delta \sigma\|_{L^2}\right)\\
& \leq C \|\nabla \Phi\|_{L^\infty} \left(\|\nabla \rho\|_{L^2}^{2 - \frac{d}{4}} + \|\nabla \sigma\|_{L^2}^{2 - \frac{d}{4}}\right) \left(\|\nabla \Delta \rho\|_{L^2}^{1 + \frac{d}{4}} + \|\nabla \Delta \sigma\|_{L^2}^{1 + \frac{d}{4}}\right)\\
& \qquad + C \|\nabla \Phi\|_{L^\infty} \left(\|\nabla \rho\|_{L^2}^2 + \|\nabla \sigma\|_{L^2}^2\right) \left(\|\nabla \Delta \rho\|_{L^2} + \|\nabla \Delta \sigma\|_{L^2}\right)\\
& \quad \leq \frac{D}{10} \left(\|\nabla \Delta \rho\|_{L^2}^2 + \|\nabla \Delta \sigma\|_{L^2}^2\right) + C \|\nabla \Phi\|_{L^\infty}^{\frac{8}{4 - d}} \left(\|\nabla \rho\|_{L^2}^{\frac{16 - 2 d}{4 - d}} + \|\nabla \sigma\|_{L^2}^{\frac{16 - 2 d}{4 - d}}\right)\\
& \qquad + C \|\nabla \Phi\|_{L^\infty}^2 \left(\|\nabla \rho\|_{L^2}^4 + \|\nabla \sigma\|_{L^2}^4\right).
\eas

To estimate the term $I_{1, 2}$, we use H\"{o}lder's inequality, Young's inequality, Ladyzhenskaya's inequality, and the bound \eqref{uLp}
\bal
I_{1, 2} & \leq \|u\|_{L^4} \left(\|\nabla \nabla \rho\|_{L^4} \|\nabla \Delta \rho\|_{L^2} + \|\nabla \nabla \sigma\|_{L^4} \|\nabla \Delta \sigma\|_{L^2}\right) \notag\\
& \leq \frac{5}{2 D} \|u\|_{L^4}^2 \left(\|\nabla \nabla \rho\|_{L^4}^2 + \|\nabla \nabla \sigma\|_{L^4}^2\right) + \frac{D}{10} \left(\|\nabla \Delta \rho\|_{L^2}^2 + \|\nabla \Delta \sigma\|_{L^2}^2\right) \notag\\
& \leq C \|u\|_{L^4}^2 \left(\|\nabla \nabla \rho\|_{L^2}^{2 - \frac{d}{2}} + \|\nabla \nabla \sigma\|_{L^2}^{2 - \frac{d}{2}}\right) \left(\|\nabla \Delta \rho\|_{L^2}^{\frac{d}{2}} + \|\nabla \Delta \sigma\|_{L^2}^{\frac{d}{2}}\right) + C \|u\|_{L^4}^2 \left(\|\nabla \nabla \rho\|_{L^2}^2 + \|\nabla \nabla \sigma\|_{L^2}^2\right) \notag\\
& \qquad + \frac{D}{10} \left(\|\nabla \Delta \rho\|_{L^2}^2 + \|\nabla \Delta \sigma\|_{L^2}^2\right) \notag\\
& \leq C \left(\|\nabla \Phi\|_{L^\infty}^2 \|\rho\|_{L^4}^2 + \|\nabla \Phi\|_{L^\infty}^{\frac{8}{4 - d}} \|\rho\|_{L^4}^{\frac{8}{4 - d}}\right) \left(\|\nabla \nabla \rho\|_{L^2}^2 + \|\nabla \nabla \sigma\|_{L^2}^2\right) \notag\\
& \qquad + \frac{D}{5} \left(\|\nabla \Delta \rho\|_{L^2}^2 + \|\nabla \Delta \sigma\|_{L^2}^2\right). \label{DtH2est2}
\eal

To estimate the term $I_{1, 3}$, we use H\"{o}lder's inequality, Ladyzhenskaya's inequality, the Gagliardo-Nirenberg interpolation inequality
\[
\|\Delta f\|_{L^2} \leq C \|\nabla \Delta f\|_{L^2}^{\frac{1}{2}} \|\nabla f\|_{L^2}^{\frac{1}{2}} + C \|\nabla f\|_{L^2},
\]
and Young's inequality
\begin{align}
\label{DtH2est4}
\begin{split}
I_{1, 3} & \leq \frac{3 D}{\ve} \|\nabla \rho\|_{L^4} \|\nabla \sigma\|_{L^4} \|\Delta \rho\|_{L^2} + \frac{3 D}{\ve} \|\nabla \rho\|_{L^4}^2 \|\Delta \sigma\|_{L^2}\\
& \leq C \left(\|\nabla \rho\|_{L^2} + \|\nabla \sigma\|_{L^2}\right) \left(\|\Delta \rho\|_{L^2}^{\frac{d}{2}} + \|\Delta \sigma\|_{L^2}^{\frac{d}{2}}\right) \left(\|\nabla \Delta \rho\|_{L^2}^{2 - \frac{d}{2}} + \|\nabla \Delta \sigma\|_{L^2}^{2 - \frac{d}{2}}\right)\\
& \qquad + C \|\nabla \rho\|_{L^2}^2 \|\nabla \sigma\|_{L^2}\\
& \leq \frac{D}{10} \left(\|\nabla \Delta \rho\|_{L^2}^2 + \|\nabla \Delta \sigma\|_{L^2}^2\right) + C \left(\|\nabla \rho\|_{L^2}^{\frac{4}{d}} + \|\nabla \sigma\|_{L^2}^{\frac{4}{d}}\right) \left(\|\Delta \rho\|_{L^2}^2 + \|\Delta \sigma\|_{L^2}^2\right)\\
& \qquad + C \|\nabla \rho\|_{L^2}^2 \|\nabla \sigma\|_{L^2}.
\end{split}
\end{align}

The estimates for the other terms in \eqref{DtH2} are similar. By H\"{o}lder's inequality, Young's inequality, the elliptic estimates, and Ladyzhenskaya's inequality, we obtain
\begin{align}
\label{DtH2est3}
\begin{split}
I_{1, 4} & \leq D \|\nabla \Phi\|_{L^\infty} \left(\|\nabla \Delta \sigma\|_{L^2} \|\Delta \rho\|_{L^2} + \|\nabla \Delta \rho\|_{L^2} \|\Delta \sigma\|_{L^2}\right)\\
& \leq \frac{5}{4} \|\nabla \Phi\|_{L^\infty}^2 \left(\|\Delta \rho\|_{L^2}^2 + \|\Delta \sigma\|_{L^2}^2\right) + \frac{D}{5} \left(\|\nabla \Delta \rho\|_{L^2}^2 + \|\nabla \Delta \sigma\|_{L^2}^2\right),
\end{split}
\end{align}
and
\begin{align}
\label{DtH2est5}
\begin{split}
I_{1, 5} & \leq C \|\nabla \nabla \Phi\|_{L^2} \|\Delta \rho\|_{L^4} \|\Delta \sigma\|_{L^4} + C \|\rho\|_{L^2} \|\Delta \rho\|_{L^4} \|\Delta \sigma\|_{L^4}\\
& \leq C \|\rho\|_{L^2} \left(\|\Delta \rho\|_{L^2}^{2 - \frac{d}{2}} + \|\Delta \sigma\|_{L^2}^{2 - \frac{d}{2}}\right) \left(\|\nabla \Delta \rho\|_{L^2}^{\frac{d}{2}} + \|\nabla \Delta \sigma\|_{L^2}^{\frac{d}{2}}\right)\\
& \leq C \|\rho\|_{L^2}^{\frac{4}{4 - d}} \left(\|\Delta \rho\|_{L^2}^2 + \|\Delta \sigma\|_{L^2}^2\right) + \frac{D}{5} \left(\|\nabla \Delta \rho\|_{L^2}^2 + \|\nabla \Delta \sigma\|_{L^2}^2\right).
\end{split}
\end{align}
Using the estimates \eqref{DtH2est1}--\eqref{DtH2est5} in \eqref{DtH2}, we conclude 
\bal
& \frac{1}{2} \frac{\diff}{\diff{t}} \left(\|\Delta \rho\|_{L^2}^2 + \|\Delta \sigma\|_{L^2}^2\right) + \frac{D}{10} \left(\|\nabla \Delta \rho\|_{L^2}^2 + \|\nabla \Delta \sigma\|_{L^2}^2\right) + \frac{D}{\ve} \int \sigma |\Delta \rho|^2 \notag\\
& \leq C \left(\|\nabla \Phi\|_{L^\infty}^2 \|\rho\|_{L^2}^2 + \|\nabla \Phi\|_{L^\infty}^{\frac{8}{4 - d}} \|\rho\|_{L^2}^{\frac{8}{4 - d}} + \|\rho\|_{L^2}^{\frac{4}{4 - d}} + \|\nabla \Phi\|_{L^\infty}^2 + \|\nabla \rho\|_{L^2}^{\frac{4}{d}} + \|\nabla \sigma\|_{L^2}^{\frac{4}{d}}\right) \left(\|\Delta \rho\|_{L^2}^2 + \|\Delta \sigma\|_{L^2}^2\right) \notag\\
& \quad + C \|\nabla \rho\|_{L^2}^2 \|\nabla \sigma\|_{L^2} + C \|\nabla \Phi\|_{L^\infty}^2 \left(\|\nabla \rho\|_{L^2}^4 + \|\nabla \sigma\|_{L^2}^4\right) + C \|\nabla \Phi\|_{L^\infty}^{\frac{8}{4 - d}} \left(\|\nabla \rho\|_{L^2}^{\frac{16 - 2 d}{4 - d}} + \|\nabla \sigma\|_{L^2}^{\frac{16 - 2 d}{4 - d}}\right). \label{DtH2est02}
\eal

For simplicity, we denote
\bax
Z & = \|\Delta \rho\|_{L^2}^2 + \|\Delta \sigma\|_{L^2}^2,\\
W_1 & = \|\nabla \Phi\|_{L^\infty}^2 \|\rho\|_{L^2}^2 + \|\nabla \Phi\|_{L^\infty}^{\frac{8}{4 - d}} \|\rho\|_{L^2}^{\frac{8}{4 - d}} + \|\rho\|_{L^2}^{\frac{4}{4 - d}} + \|\nabla \Phi\|_{L^\infty}^2 + \|\nabla \rho\|_{L^2}^{\frac{4}{d}} + \|\nabla \sigma\|_{L^2}^{\frac{4}{d}},\\
W_2 & = \|\nabla \rho\|_{L^2}^2 \|\nabla \sigma\|_{L^2} + \|\nabla \Phi\|_{L^\infty}^2 \left(\|\nabla \rho\|_{L^2}^4 + \|\nabla \sigma\|_{L^2}^4\right) + \|\nabla \Phi\|_{L^\infty}^{\frac{8}{4 - d}} \left(\|\nabla \rho\|_{L^2}^{\frac{16 - 2 d}{4 - d}} + \|\nabla \sigma\|_{L^2}^{\frac{16 - 2 d}{4 - d}}\right).
\eax
From \eqref{solest1}(iii), it follows that
\begin{align}
\label{Ztimeintbdd}
\int_0^t Z(\tau) \diff{\tau} \leq C.
\end{align}
By \eqref{solest1}(i)--(iii), we have that
\begin{align}
\sup_{\tau \in [0, t]} W_1(\tau) \leq C,
\end{align}
and that
\bas{W2bdd}
\int_0^t W_2(\tau) \diff{\tau} & \leq \sup_{\tau \in [0, t]} \|\nabla \sigma(\tau)\|_{L^2} \cdot \int_0^t \|\nabla \rho(\tau)\|_{L^2}^2 \diff{\tau}\\
& \qquad + \sup_{\tau \in [0, t]} \left(\|\nabla \rho(\tau)\|_{L^2}^4 + \|\nabla \sigma(\tau)\|_{L^2}^4\right) \cdot \int_0^t \|\nabla \Phi(\tau)\|_{L^\infty}^2 \diff{\tau}\\
& \qquad + \sup_{\tau \in [0, t]} \left(\|\nabla \rho(\tau)\|_{L^2}^{\frac{16 - 2 d}{4 - d}} + \|\nabla \sigma(\tau)\|_{L^2}^{\frac{16 - 2 d}{4 - d}}\right) \cdot \int_0^t \|\nabla \Phi(\tau)\|_{L^\infty}^{\frac{8}{4 - d}} \diff{\tau}\\
& \leq C.
\eas
Using the fact that $\sigma \geq 0$ and applying the bounds \eqref{Ztimeintbdd}--\eqref{W2bdd}, the inequality \eqref{DtH2est02} leads to \eqref{deltal2bdd}.
\end{proof}

Next, we propagate the $H^3$ regularity of the solutions.
\begin{lemma}
Let $d = 2, 3$ and $r \geq 2$. Let $c_1(0), c_2(0) \in H^3(\T^d)$ be nonnegative functions satisfying \eqref{neut}. Suppose $(\rho, \sigma, u)$ solves \eqref{rhot}--\eqref{incomp} with initial data \eqref{initial} on the interval $[0, T]$. Then for any $t \in [0, T]$, we have
\bal
& \|\nabla \Delta \rho(t)\|_{L^2} + \|\nabla \Delta \sigma(t)\|_{L^2} + \int_0^t \|\Delta^2 \rho(\tau)\|_{L^2}^2 + \|\Delta^2 \sigma(\tau)\|_{L^2}^2 \leq C, \label{rhosigHsbdd}\\
& \|u\|_{H^3} \leq C, \label{uHsbdd}
\eal
for some constants $C > 0$.
\end{lemma}
\begin{proof}
We multiply \eqref{rhot} and \eqref{sigmat0} by $- \Delta^3 \rho$ and $- \Delta^3 \sigma$ respectively, integrate over $\T^d$. We integrate by parts and use \eqref{poisson} to obtain
\begin{align}
\label{DtH3}
\begin{split}
& \frac{1}{2} \frac{\diff}{\diff{t}} \left(\|\nabla \Delta \rho\|_{L^2}^2 + \|\nabla \Delta \sigma\|_{L^2}^2\right) + D \left(\|\Delta^2 \rho\|_{L^2}^2 + \|\Delta^2 \sigma\|_{L^2}^2\right) + \frac{D}{\ve} \int \sigma |\nabla \Delta \rho|^2\\
&\quad= I_{2, 1} + I_{2, 2} + I_{2, 3} + I_{2, 4} + I_{2, 5} + I_{2, 6},
\end{split}
\end{align}
where
\begin{align*}
I_{2, 1} & = \int \Delta u \cdot (\nabla \rho \Delta^2 \rho + \nabla \sigma \Delta^2 \sigma) + \int u \cdot (\nabla \Delta \rho \Delta^2 \rho + \nabla \Delta \sigma \Delta^2 \sigma) + 2 \int \nabla u : (\nabla \nabla \rho \Delta^2 \rho + \nabla \nabla \sigma \Delta^2 \sigma),\\
I_{2, 2} & = D \int \nabla \Delta \rho \cdot \nabla \nabla \Delta \sigma \nabla \Phi + D \int \nabla \Delta \sigma \cdot \nabla \nabla \Delta \rho \nabla \Phi,\\
I_{2, 3} & = - \frac{2 D}{\ve} \int \nabla \Delta \rho \cdot (\nabla \nabla \rho \nabla \sigma) - \frac{5 D}{\ve} \int \nabla \Delta \sigma \cdot (\nabla \nabla \rho \nabla \rho) - \frac{3 D}{\ve} \int \nabla \Delta \rho \cdot (\nabla \nabla \sigma \nabla \rho)\\
& \quad - \frac{D}{\ve} \int \nabla \Delta \rho \cdot \nabla \rho \Delta \sigma - \frac{3 D}{\ve} \int \nabla \Delta \rho \cdot \nabla \sigma \Delta \rho - \frac{4 D}{\ve} \int \nabla \Delta \sigma \cdot \nabla \rho \Delta \rho,\\
I_{2, 4} & = 2 D \int \nabla \rho \cdot (\nabla \nabla \nabla \sigma : \nabla \nabla \Phi) + 2 D \int \nabla \sigma \cdot (\nabla \nabla \nabla \rho : \nabla \nabla \Phi),\\
I_{2, 5} & = - \frac{3 D}{\ve} \int \nabla \Delta \rho \cdot \nabla \Delta \sigma \rho,\\
I_{2, 6} & = 2 D \int \nabla \rho \cdot (\nabla \nabla \nabla \Phi : \nabla \nabla \sigma) + 2 D \int \nabla \sigma \cdot (\nabla \nabla \nabla \Phi : \nabla \nabla \rho).
\end{align*}

First, from \eqref{darcy} and the fact that the Leray projector commutes with the Laplacian, we find that
\bas{Delu}
\|\Delta u\|_{L^2} \leq C \|\Delta (\rho \nabla \Phi)\|_{L^2} \leq C \|\Delta \rho\|_{L^2} \|\nabla \Phi\|_{L^\infty} + C \|\nabla \rho\|_{L^6} \|\rho\|_{L^3},
\eas
where we also used H\"{o}lder's inequality and the equation \eqref{poisson} in the second inequality.

For the terms involving velocity $u$, we use H\"{o}lder's inequalities for $L^2$-$L^2$-$L^\infty$ or $L^2$-$L^3$-$L^6$, the estimate \eqref{nabuLr}, the Gagliardo-Nirenberg interpolation inequality
\[
\|\nabla \nabla f\|_{L^3} \leq C \|\Delta^2 f\|_{L^2}^{\frac{d}{12}} \|\Delta f\|_{L^2}^{1 - \frac{d}{12}} + C \|\Delta f\|_{L^2},
\]
the embeddings $H^2(\T^d) \hookrightarrow L^\infty(\T^d)$ and $H^1(\T^d) \hookrightarrow L^6(\T^d)$, $d = 2, 3$, the bound \eqref{Delu}, and Young's inequality,
\bal
I_{2, 1} & \leq \|\Delta u\|_{L^2} \left(\|\nabla \rho\|_{L^\infty} \|\Delta^2 \rho\|_{L^2} + \|\nabla \sigma\|_{L^\infty} \|\Delta^2 \sigma\|_{L^2}\right) + \|u\|_{L^\infty} \left(\|\Delta^2 \rho\|_{L^2} \|\nabla \Delta \rho\|_{L^2} + \|\Delta^2 \sigma\|_{L^2} \|\nabla \Delta \sigma\|_{L^2}\right) \notag\\
& \quad + 2 \|\nabla u\|_{L^6} \left(\|\nabla \nabla \rho\|_{L^3} \|\Delta^2 \rho\|_{L^2} + \|\nabla \nabla \sigma\|_{L^3} \|\Delta^2 \sigma\|_{L^2}\right) \notag\\
& \leq \|\Delta u\|_{L^2} \left(\|\nabla \rho\|_{L^\infty} \|\Delta^2 \rho\|_{L^2} + \|\nabla \sigma\|_{L^\infty} \|\Delta^2 \sigma\|_{L^2}\right) + \|u\|_{L^\infty} \left(\|\Delta^2 \rho\|_{L^2} \|\nabla \Delta \rho\|_{L^2} + \|\Delta^2 \sigma\|_{L^2} \|\nabla \Delta \sigma\|_{L^2}\right) \notag\\
& \quad + C \|\nabla \rho\|_{L^6} \|\nabla \Phi\|_{L^\infty} \left(\|\Delta \rho\|_{L^2}^{1 - \frac{d}{12}} + \|\Delta \sigma\|_{L^2}^{1 - \frac{d}{12}}\right) \left(\|\Delta^2 \rho\|_{L^2}^{1 + \frac{d}{12}} + \|\Delta^2 \sigma\|_{L^2}^{1 + \frac{d}{12}}\right) \notag\\
& \quad + C \|\nabla \rho\|_{L^6} \|\nabla \Phi\|_{L^\infty} \left(\|\Delta \rho\|_{L^2} + \|\Delta \sigma\|_{L^2}\right) \left(\|\Delta^2 \rho\|_{L^2}+ \|\Delta^2 \sigma\|_{L^2}\right) \notag\\
& \leq \frac{D}{5} \left(\|\Delta^2 \rho\|_{L^2}^2 + \|\Delta^2 \sigma\|_{L^2}^2\right) + C \|\Delta u\|_{L^2}^2 \left(\|\nabla \rho\|_{L^\infty}^2 + \|\nabla \sigma\|_{L^\infty}^2\right) + C \|u\|_{H^2}^2 \left(\|\nabla \Delta \rho\|_{L^2}^2 + \|\nabla \Delta \sigma\|_{L^2}^2\right) \notag\\
& \quad + C \left(\|\nabla \rho\|_{L^6}^{\frac{24}{12 - d}} \|\nabla \Phi\|_{L^\infty}^{\frac{24}{12 - d}} + \|\nabla \rho\|_{L^6}^2 \|\nabla \Phi\|_{L^\infty}^2\right) \left(\|\Delta \rho\|_{L^2}^2 + \|\Delta \sigma\|_{L^2}^2\right) \notag\\
& \leq \frac{D}{5} \left(\|\Delta^2 \rho\|_{L^2}^2 + \|\Delta^2 \sigma\|_{L^2}^2\right) + C \left(\|\Delta \rho\|_{L^2}^2 \|\nabla \Phi\|_{L^\infty}^2 + \|\rho\|_{H^2}^2 \|\rho\|_{L^3}^2\right) \left(\|\nabla \Delta \rho\|_{L^2}^2 + \|\nabla \Delta \sigma\|_{L^2}^2\right) \notag\\
& \quad + C \left(\|\rho\|_{H^2}^{\frac{24}{12 - d}} \|\nabla \Phi\|_{L^\infty}^{\frac{24}{12 - d}} + \|\rho\|_{H^2}^2 \|\nabla \Phi\|_{L^\infty}^2\right) \left(\|\Delta \rho\|_{L^2}^2 + \|\Delta \sigma\|_{L^2}^2\right). \label{DtH3est1}
\eal
For the term $I_{2, 2}$, we use H\"{o}lder's inequality and Young's inequality to get
\begin{align}
\begin{split}
I_{2, 2} & \leq D \|\nabla \Phi\|_{L^\infty} \left(\|\nabla \Delta \rho\|_{L^2} \|\nabla \nabla \Delta \sigma\|_{L^2} + \|\nabla \Delta \sigma\|_{L^2} \|\nabla \nabla \Delta \rho\|_{L^2}\right)\\
& \leq C \|\nabla \Phi\|_{L^\infty}^2 \left(\|\nabla \Delta \rho\|_{L^2}^2 + \|\nabla \Delta \sigma\|_{L^2}^2\right) + \frac{D}{5} \left(\|\Delta^2 \rho\|_{L^2}^2 + \|\Delta^2 \sigma\|_{L^2}^2\right).
\end{split}
\end{align}
By H\"{o}lder's inequality for $L^2$-$L^4$-$L^4$ and the Gagliardo-Nirenberg inequalities
\begin{align*}
& \|\nabla \Delta f\|_{L^2} \leq C \|\Delta^2 f\|_{L^2}^{\frac{1}{2}} \|\Delta f\|_{L^2}^{\frac{1}{2}} + C \|\Delta f\|_{L^2},\\
& \|\Delta f\|_{L^4} \leq C \|\Delta^2 f\|_{L^2}^{\frac{d + 4}{12}} \|\nabla f\|_{L^2}^{\frac{8 - d}{12}} + C \|\nabla f\|_{L^2},\\
& \|\nabla f\|_{L^4} \leq C \|\Delta^2 f\|_{L^2}^{\frac{d}{12}} \|\nabla f\|_{L^2}^{1 - \frac{d}{12}} + C \|\nabla f\|_{L^2},
\end{align*}
we obtain
\bax
I_{2, 3} & \leq C \|\nabla \Delta \rho\|_{L^2} \|\Delta \rho\|_{L^4} \|\nabla \sigma\|_{L^4} + C \|\nabla \Delta \sigma\|_{L^2} \|\Delta \rho\|_{L^4} \|\nabla \rho\|_{L^4} + C \|\nabla \Delta \rho\|_{L^2} \|\Delta \sigma\|_{L^4} \|\nabla \rho\|_{L^4}\\
& \leq C \left(\|\Delta^2 \rho\|_{L^2}^{\frac{1}{2}} \|\Delta \rho\|_{L^2}^{\frac{1}{2}} + \|\Delta \rho\|_{L^2}\right) \left(\|\Delta^2 \rho\|_{L^2}^{\frac{d + 4}{12}} \|\nabla \rho\|_{L^2}^{\frac{8 - d}{12}} + \|\nabla \rho\|_{L^2}\right) \left(\|\Delta^2 \sigma\|_{L^2}^{\frac{d}{12}} \|\nabla \sigma\|_{L^2}^{1 - \frac{d}{12}} + \|\nabla \sigma\|_{L^2}\right)\\
& \quad + C \left(\|\Delta^2 \sigma\|_{L^2}^{\frac{1}{2}} \|\Delta \sigma\|_{L^2}^{\frac{1}{2}} + \|\Delta \sigma\|_{L^2}\right) \left(\|\Delta^2 \rho\|_{L^2}^{\frac{d + 4}{12}} \|\nabla \rho\|_{L^2}^{\frac{8 - d}{12}} + \|\nabla \rho\|_{L^2}\right) \left(\|\Delta^2 \rho\|_{L^2}^{\frac{d}{12}} \|\nabla \rho\|_{L^2}^{1 - \frac{d}{12}} + \|\nabla \rho\|_{L^2}\right)\\
& \quad + C \left(\|\Delta^2 \rho\|_{L^2}^{\frac{1}{2}} \|\Delta \rho\|_{L^2}^{\frac{1}{2}} + \|\Delta \rho\|_{L^2}\right) \left(\|\Delta^2 \sigma\|_{L^2}^{\frac{d + 4}{12}} \|\nabla \sigma\|_{L^2}^{\frac{8 - d}{12}} + \|\nabla \sigma\|_{L^2}\right) \left(\|\Delta^2 \rho\|_{L^2}^{\frac{d}{12}} \|\nabla \rho\|_{L^2}^{1 - \frac{d}{12}} + \|\nabla \rho\|_{L^2}\right)\\
& \leq C \left(\|\Delta^2 \rho\|_{L^2}^{\frac{5 + d}{6}} + \|\Delta^2 \sigma\|_{L^2}^{\frac{5 + d}{6}}\right) \left( \|\Delta \rho\|_{L^2}^{\frac{1}{2}} + \|\Delta \sigma\|_{L^2}^{\frac{1}{2}}\right) \left( \|\nabla \rho\|_{L^2}^{\frac{10 - d}{6}} + \|\nabla \sigma\|_{L^2}^{\frac{10 - d}{6}}\right)\\
& \quad + C \left(\|\nabla \rho\|_{L^2}^2 + \|\nabla \sigma\|_{L^2}^2\right) \left(\|\Delta \rho\|_{L^2} + \|\Delta \sigma\|_{L^2}\right),
\eax
which, by Young's inequality, implies
\begin{align}
\label{DtH3est3}
\begin{split}
I_{2, 3}
& \leq \frac{D}{5} \left(\|\Delta^2 \rho\|_{L^2}^2 + \|\Delta^2 \sigma\|_{L^2}^2\right) + C \left(\|\nabla \rho\|_{L^2}^{\frac{20 - 2 d}{7 - d}} + \|\nabla \sigma\|_{L^2}^{\frac{20 - 2 d}{7 - d}}\right) \left(\|\Delta \rho\|_{L^2}^{\frac{6}{7 - d}} + \|\Delta \sigma\|_{L^2}^{\frac{6}{7 - d}}\right)\\
&\qquad\qquad+ C \left(\|\nabla \rho\|_{L^2}^2 + \|\nabla \sigma\|_{L^2}^2\right) \left(\|\Delta \rho\|_{L^2} + \|\Delta \sigma\|_{L^2}\right).
\end{split}
\end{align}
The estimates for $I_{2, 4}$ and $I_{2, 5}$ follow from H\"{o}lder's inequality, the Sobolev embedding $H^1(\T^d) \hookrightarrow L^6(\T^d)$, and Young's inequality,
\begin{align}
\label{DtH3est4}
\begin{split}
I_{2, 4} & \leq C \|\rho\|_{L^3} \left(\|\nabla \rho\|_{L^6} \|\nabla \Delta \sigma\|_{L^2} + \|\nabla \sigma\|_{L^6} \|\nabla \Delta \rho\|_{L^2}\right)\\
& \leq C \|\rho\|_{L^3} \left(\|\rho\|_{H^2} + \|\sigma - \bar{\sigma}\|_{H^2}\right) \left(\|\nabla \Delta \rho\|_{L^2} + \|\nabla \Delta \sigma\|_{L^2}\right)\\
& \leq C \|\rho\|_{L^3}^2 \left(\|\rho\|_{H^2}^2 + \|\sigma - \bar{\sigma}\|_{H^2}^2\right) + C \left(\|\nabla \Delta \rho\|_{L^2}^2 + \|\nabla \Delta \sigma\|_{L^2}^2\right),
\end{split}
\end{align}
and
\begin{align}
\label{DtH3est5}
\begin{split}
I_{2, 5} & \leq C \|\rho\|_{L^\infty} \|\nabla \Delta \rho\|_{L^2} \|\nabla \Delta \sigma\|_{L^2} \leq C \|\rho\|_{H^2} \|\nabla \Delta \rho\|_{L^2} \|\nabla \Delta \sigma\|_{L^2}.
\end{split}
\end{align}
Finally, we use H\"{o}lder's inequality and Ladyzhenskaya's inequality to obtain
\bal
I_{2, 6} & \leq C \|\nabla \rho\|_{L^4}^2 \|\Delta \sigma\|_{L^2} + C \|\nabla \sigma\|_{L^4} \|\nabla \rho\|_{L^4} \|\Delta \rho\|_{L^2} \notag\\
& \leq C \left(\|\nabla \rho\|_{L^2}^2 + \|\nabla \sigma\|_{L^2}^2\right) \left(\|\Delta \rho\|_{L^2} + \|\Delta \sigma\|_{L^2}\right) \notag\\
& \quad + C \left(\|\nabla \rho\|_{L^2}^{2 - \frac{d}{2}} + \|\nabla \sigma\|_{L^2}^{2 - \frac{d}{2}}\right) \left(\|\Delta \rho\|_{L^2}^{1 + \frac{d}{2}} + \|\Delta \sigma\|_{L^2}^{1 + \frac{d}{2}}\right). \label{DtH3est6}
\eal
Gathering the estimates \eqref{DtH3est1}--\eqref{DtH3est6} into \eqref{DtH3}, we arrive at
\bal
& \frac{1}{2} \frac{\diff}{\diff{t}} \left(\|\nabla \Delta \rho\|_{L^2}^2 + \|\nabla \Delta \sigma\|_{L^2}^2\right) + \frac{D}{5} \left(\|\Delta^2 \rho\|_{L^2}^2 + \|\Delta^2 \sigma\|_{L^2}^2\right) + \frac{D}{\ve} \int \sigma |\nabla \Delta \rho|^2 \notag\\
& \quad \leq C \left(1 + \|\rho\|_{H^2} + \|\nabla \Phi\|_{L^\infty}^2 + \|\Delta \rho\|_{L^2}^2 \|\nabla \Phi\|_{L^\infty}^2 + \|\rho\|_{H^2}^2 \|\rho\|_{L^3}^2\right) \left(\|\nabla \Delta \rho\|_{L^2}^2 + \|\nabla \Delta \sigma\|_{L^2}^2\right) \notag\\
& \qquad + C \|\rho\|_{L^3}^2 \left(\|\rho\|_{H^2}^2 + \|\sigma - \bar{\sigma}\|_{H^2}^2\right) + C \left(\|\nabla \rho\|_{L^2}^{2 - \frac{d}{2}} + \|\nabla \sigma\|_{L^2}^{2 - \frac{d}{2}}\right) \left(\|\Delta \rho\|_{L^2}^{1 + \frac{d}{2}} + \|\Delta \sigma\|_{L^2}^{1 + \frac{d}{2}}\right) \notag\\
& \qquad + C \left(\|\nabla \rho\|_{L^2}^{\frac{20 - 2d}{7 - d}} + \|\nabla \sigma\|_{L^2}^{\frac{20 - 2d}{7 - d}}\right) \left(\|\Delta \rho\|_{L^2}^{\frac{6}{7 - d}} + \|\Delta \sigma\|_{L^2}^{\frac{6}{7 - d}}\right) \notag\\
& \qquad + C \left(\|\rho\|_{H^2}^{\frac{24}{12 - d}} \|\nabla \Phi\|_{L^\infty}^{\frac{24}{12 - d}} + \|\rho\|_{H^2}^2 \|\nabla \Phi\|_{L^\infty}^2\right) \left(\|\Delta \rho\|_{L^2}^2 + \|\Delta \sigma\|_{L^2}^2\right) \notag\\
& \qquad + C \left(\|\nabla \rho\|_{L^2}^2 + \|\nabla \sigma\|_{L^2}^2\right) \left(\|\Delta \rho\|_{L^2} + \|\Delta \sigma\|_{L^2}\right). \label{DtH3estlast}
\eal
We drop the dissipation terms, integrate in time, and use the bounds \eqref{solest1}(i)--(iii) and \eqref{deltal2bdd} to obtain
\[
\|\nabla \Delta \rho(t)\|_{L^2}^2 + \|\nabla \Delta \sigma(t)\|_{L^2}^2 \leq C.
\]
Going back to \eqref{DtH3estlast}, we conclude that
\[
\int_0^t \|\Delta^2 \rho(\tau)\|_{L^2} + \|\Delta^2 \sigma(\tau)\|_{L^2} \diff{\tau} \leq C.
\]
Finally, the Leibnitz rule, H\"{o}lder's inequality, and \eqref{rhosigHsbdd} conclude \eqref{uHsbdd}.
\end{proof}

The last lemma establishes the uniqueness of strong solutions.
\begin{lemma}
Let $d = 2, 3$ and $s > 1 + \frac{d}{2}$. Let $(\rho_1, \sigma_1)$, $(\rho_2, \sigma_2) \in \left(L^\infty([0, T]; H^s(\T^d))\right)^2$ be two solutions of the initial value problem \eqref{rhot}--\eqref{initial}. Then $(\rho_1, \sigma_1) = (\rho_2, \sigma_2)$.
\end{lemma}
\begin{proof}
We denote the differences of $(\rho_1, \sigma_1)$ and $(\rho_2, \sigma_2)$ by $(\tilde{\rho}, \tilde{\sigma}) = (\rho_1 - \rho_2, \sigma_1 - \sigma_2)$. Then the equations for $(\tilde{\rho}, \tilde{\sigma})$ are
\begin{align*}
& \pt \tilde{\rho} = - u_1 \cdot \nabla \tilde{\rho} - \tilde{u} \cdot \nabla \rho_2 + D \Delta \tilde{\rho} + D (\nabla \sigma_1 \cdot \nabla \tilde{\Phi} + \nabla \tilde{\sigma} \cdot \nabla \Phi_2) + D (\sigma_1 \Delta \tilde{\Phi} + \tilde{\sigma} \Delta \Phi_2),\\
& \pt \tilde{\sigma} = - u_1 \cdot \nabla \tilde{\sigma} - \tilde{u} \cdot \nabla \sigma_2 + D \Delta \tilde{\sigma} + D (\nabla \rho_1 \cdot \nabla \tilde{\Phi} + \nabla \tilde{\rho} \cdot \nabla \Phi_2) + D (\rho_1 \Delta \tilde{\Phi} + \tilde{\rho} \Delta \Phi_2),\\
& - \ve \Delta \tilde{\Phi} = \tilde{\rho},\\
& \tilde{u} + \nabla (p_1 - p_2) = - \tilde{\rho} \nabla \Phi_1 - \rho_2 \nabla \tilde{\Phi},\\
& \nabla \cdot \tilde{u} = 0,
\end{align*}
with initial data
\[
\tilde{\rho}(\cdot, 0) = \tilde{\sigma}(\cdot, 0) = 0.
\]
The $L^2$-estimates for $(\tilde{\rho}, \tilde{\sigma}, \tilde{u})$ lead to
\begin{align}
\label{enuniq}
\begin{split}
& \frac{1}{2} \frac{\diff}{\diff{t}} \left(\|\tilde{\rho}\|_{L^2}^2 + \|\tilde{\sigma}\|_{L^2}^2\right) + D \left(\|\nabla \tilde{\rho}\|_{L^2}^2 + \|\nabla \tilde{\sigma}\|_{L^2}^2\right) + \frac{D}{\ve} \int \tilde{\rho}^2 \sigma_1\\
&\quad= - \int \tilde{\rho} \tilde{u} \cdot \nabla \rho_2 - \int \tilde{\sigma} \tilde{u} \cdot \nabla \sigma_2 + D \int \tilde{\rho} \nabla \sigma_1 \cdot \nabla \tilde{\Phi} + D \int \tilde{\sigma} \nabla \rho_1 \cdot \nabla \tilde{\Phi} - \frac{D}{\ve} \int \tilde{\rho} \tilde{\sigma} (\rho_1 + \rho_2)\\
&\quad\leq C \left(\|\nabla \rho_2\|_{L^\infty} + \|\nabla \sigma_2\|_{L^\infty} + \|\rho_1\|_{L^\infty} + \|\rho_2\|_{L^\infty} \right) \left(\|\tilde{\rho}\|_{L^2}^2 + \|\tilde{\sigma}\|_{L^2}^2\right).
\end{split}
\end{align}
Because $\sigma_1 \geq 0$, we can drop the last two terms in the first line of \eqref{enuniq}. Finally, in view of Sobolev embeddings and Gr\"{o}nwall's inequality, we obtain $(\tilde{\rho}, \tilde{\sigma}) = (0, 0)$, which proves the uniqueness.
\end{proof}


\begin{thebibliography}{99}

\bibitem{AI20p} \textsc{E.~Abdo and M.~Ignatova}. Long time finite dimensionality in charged fluids. Preprint, 2020.

\bibitem{BFS14} \textsc{D.~Bothe, A.~Fischer, and J.~Saal}. Global well-posedness and stability of electrokinetic flows. {\it SIAM J. Math. Anal.}, {\bf 46}(2), 1263--1316, 2014.

\bibitem{CL95} \textsc{Y.~S.~Choi and R.~Lui}. Multi-dimensional electrochemistry model. {\it Arch. Ration. Mech. Anal.}, {\bf 130}(4), 315--342, 1995.

\bibitem{CEIV17} \textsc{P.~Constantin, T.~Elgindi, M.~Ignatova, and V.~Vicol}. On some electroconvection models. {\it J. Nonlinear Sci.}, {\bf 27}(1), 197--211, 2017.

\bibitem{CF88} \textsc{P.~Constantin and C.~Foias}. {\it {N}avier-{S}tokes Equations}, Chicago Lectures in Mathematics, University of Chicago Press, Chicago, IL, 1988.

\bibitem{CI19} \textsc{P.~Constantin and M.~Ignatova}. On the {N}ernst--{P}lanck--{N}avier--{S}tokes system. {\it Arch. Rational Mech. Anal.}, {\bf 232}(3), 1379--1428, 2019.

\bibitem{CIL20pc} \textsc{P.~Constantin, M.~Ignatova, and F.-N.~Lee}. Interior electroneutrality in {N}ernst--{P}lanck--{N}avier--{S}tokes systems. Preprint arXiv:2011.15057.

\bibitem{CIL20pa} \textsc{P.~Constantin, M.~Ignatova, and F.-N.~Lee}. {N}ernst--{P}lanck--{N}avier--{S}tokes systems near equilibrium. Preprint arXiv:2008.10440.

\bibitem{CIL20pb} \textsc{P.~Constantin, M.~Ignatova, and F.-N.~Lee}. {N}ernst--{P}lanck--{N}avier--{S}tokes systems far from equilibrium. {\it Arch. Rational Mech. Anal.}, {\bf 240}(2), 1147--1168, 2021.

\bibitem{FS17} \textsc{A.~Fischer and J.~Saal}. Global weak solutions in three space dimensions for electrokinetic flow processes. {\it J. Evol. Equ.}, {\bf 17}(1), 309--333, 2017.

\bibitem{HK16pa} \textsc{M.~Herz and P.~Knabner}. Global existence of weak solutions of a model for electrolyte solutions - Part 1: two-component case. Preprint arXiv:1605.07396.

\bibitem{HK16pb} \textsc{M.~Herz and P.~Knabner}. Global existence of weak solutions of a model for electrolyte solutions - Part 2: multicomponent case. Preprint arXiv:1605.07445.

\bibitem{HRK12} \textsc{M.~Herz, N.~Ray, and P.~Knabner}. Existence and uniqueness of a global weak solution of a {D}arcy-{N}ernst-{P}lanck-{P}oisson system. {\it GAMM-Mitt.}, {\bf 35}(2), 191--208, 2012.

\bibitem{IS21p} \textsc{M.~Ignatova and J.~Shu}. Global solutions of the {N}ernst-{P}lanck-{E}uler equations. {\it SIAM J. Math. Anal.}, to appear.

\bibitem{Jer02} \textsc{J.~W.~Jerome}. Analytical approaches to charge transport in a moving medium. {\it Transport Theory Statist. Phys.}, {\bf 31}(4-6), 333--366, 2002.

\bibitem{JS09} \textsc{J.~W.~Jerome and R.~Sacco}. Global weak solutions for an incompressible charged fluid with multi-scale couplings: initial-boundary-value problem. {\it Nonlinear Anal.}, {\bf 71}(12), e2487--e2497, 2009.

\bibitem{Lee21p} \textsc{F.-N.~Lee}. Global regularity for {N}ernst-{P}lanck-{N}avier-{S}tokes systems. Preprint arXiv:2106.01569.

\bibitem{Li09} \textsc{F.~Li}. Quasineutral limit of the electro-diffusion model arising in electrohydrodynamics. {\it J. Differential Equations}, {\bf 246}(9), 3620--3641, 2009.

\bibitem{LW20} \textsc{J.-G.~Liu and J.~Wang}. Global existence for {N}ernst--{P}lanck--{N}avier--{S}tokes system in $\mathbb{R}^n$. {\it Comm. Math. Sci.}, {\bf 18}(6), 1743--1754, 2020.

\bibitem{Rub90} \textsc{I.~Rubinstein}. {\it Electro-Diffusion of Ions}, SIAM Studies in Applied Mathematics, {\bf 11}, Society for Industrial and Applied Mathematics (SIAM), Philadelphia, PA, 1990.

\bibitem{Ryh09p} \textsc{R.~J.~Ryham}. Existence, uniqueness, regularity and long-term behavior for dissipative systems modeling electrohydrodynamics. Preprint arXiv:0910.4973.

\bibitem{Sch09} \textsc{M.~Schmuck}. Analysis of the {N}avier--{S}tokes--{N}ernst--{P}lanck--{P}oisson system. {\it Math. Models Methods Appl. Sci.}, {\bf 19}(6), 993--1015, 2009.

\bibitem{WJ21} \textsc{S.~Wang and L.~Jiang}. Quasi-neutral limit and the initial layer problem of the electro-diffusion model arising in electro-hydrodynamics. {\it Nonlinear Anal. Real World Appl.}, {\bf 59}, 103266, 25, 2021.

\bibitem{WJL19} \textsc{S.~Wang, L.~Jiang, and C.~Liu}. Quasi-neutral limit and the boundary layer problem of Planck-Nernst-Poisson-Navier-Stokes equations for electro-hydrodynamics. {\it J. Differential Equations}, {\bf 267}, 3475--3523, 2019.

\bibitem{ZY15} \textsc{Z.~Zhang and Z.~Yin}. Global well-posedness for the {E}uler-{N}ernst-{P}lanck-{P}oisson system in dimension two. {\it Nonlinear Anal.} {\bf 125}, 30--53, 2015.

\bibitem{ZY15b} \textsc{Z.~Zhang and Z.~Yin}. Global well-posedness for the {N}avier-{S}tokes-{N}ernst-{P}lanck-{P}oisson system in dimension two. {\it Appl. Math. Lett.}, {\bf 40}, 102--106, 2015.

\bibitem{ZY20} \textsc{Z.~Zhang and Z.~Yin}. The inviscid limit and well-posedness for the {E}uler-{N}ernst-{P}lanck-{P}oisson system. {\it Appl. Anal.} {\bf 99}(2), 181--213, 2020

\end{thebibliography}
\end{document}